\documentclass[onefignum,onetabnum]{siamonline220329}

\usepackage{amsmath,amssymb}
\usepackage{bm}
\usepackage{graphicx,color}
\usepackage{mathtools}
\usepackage{algorithm}
\usepackage{algorithmic}
\newtheorem{assumption}{Assumption}
\newtheorem{example}{Example}
\newtheorem{remark}{Remark}
\newcommand{\df}[0]{\kappa}

\newcommand{\E}[1]{\mathbb{E}\LQ #1\RQ}
\newcommand{\EQ}[1]{\mathbb{E}^{\betab}\LQ #1\RQ}

\newcommand{\Ldue}{L^2(\mathcal{D})}

\newcommand{\LQ}[0]{\left[}
\newcommand{\RQ}[0]{\right]}
\newcommand{\kb}{{\bm k}}
\newcommand{\unob}{\bm{1}}

\newcommand{\jb}{\bm{j}}

\newcommand{\rb}{\bm{r}}

\newcommand{\I}{\mathcal{I}}
\newcommand{\zetab}{\bm{\zeta}}

\newcommand{\eb}[1]{\mathbf{e}_{#1}}

\newcommand{\ib}[0]{\mathbf{i}}
\newcommand{\alphab}[0]{{\bm \alpha}}

\newcommand{\betab}[0]{{\bm \beta}}

\newcommand{\error}{\text{error}}
\newcommand{\stoc}{\text{stoc}}
\newcommand{\dete}{\text{det}}
\newcommand{\work}{\text{work}}
\newcommand{\gb}{\bm{ g}}
\newcommand{\gammab}{\bm{ \gamma}}
\newcommand{\Deltab}{\bm{\Delta}}

\definecolor{Tommaso}{rgb}{0, 0, 1}
\definecolor{Tommaso2}{rgb}{1, 0, 0}
\definecolor{Fabio}{rgb}{1,1,0}

\begin{document}

\begin{center}

\begin{Large}
\textbf{A combination technique for optimal control problems constrained by random PDEs}
\end{Large}

\medskip

Fabio Nobile$^{1}$ and Tommaso Vanzan$^{2}$
\medskip

${}^1$ CSQI Chair, \'Ecole Polytechnique F\'ed\'erale de Lausanne, Switzerland (fabio.nobile@epfl.ch).\\
${}^2$ Dipartimento di Scienze Matematiche, Politecnico di Torino, Italy (tommaso.vanzan@polito.it).

\end{center}
\begin{abstract}
We present a combination technique based on mixed differences of both spatial approximations and quadrature formulae for the stochastic variables to solve efficiently a class of Optimal Control Problems (OCPs) constrained by random partial differential equations.
The method requires to solve the OCP for several low-fidelity spatial grids and quadrature formulae for the objective functional. 
All the computed solutions are then linearly combined to get a final approximation which, under suitable regularity assumptions, preserves the same accuracy of fine tensor product approximations, while drastically reducing the computational cost. The combination technique involves only tensor product quadrature formulae, thus the discretized OCPs preserve the (possible) convexity of the continuous OCP. Hence, the combination technique avoids the inconveniences of Multilevel Monte Carlo and/or sparse grids approaches, but remains suitable for high dimensional problems. The manuscript presents an a-priori procedure to choose the most important mixed differences and an analysis stating that the asymptotic complexity is exclusively determined by the spatial solver. Numerical experiments validate the results.
\end{abstract}

\medskip

\medskip

\section{Introduction}
In this work, we propose a new framework to discretize and solve Optimal Control Problems (OCPs) constrained by random Partial Differential Equations (PDE).
This class of problems is increasingly more popular in the design of complex engineering systems, since a full knowledge of the physical PDE-model is often not available, and the associated uncertainty is frequently modelled through random parameters. Consequently, the topic has drawn much attention in the last decade, see, e.g, the monographs \cite{kouri2018optimization,martinez2018optimal} and references there in.

We consider the following optimization problem with a random PDE-constraint depending on a random vector $\zetab=(\zeta_1,\zeta_2,\dots,\zeta_N)$ taking values in a compact set $\Gamma\subset \mathbb{R}^N$ and with density $\rho$, 
\begin{equation}\label{eq:OCP_model_intro}
\begin{cases}
\min_{u\in U} \E{F(y(\zetab))}+\frac{\nu}{2}\|u\|^2_{U}\\
\text{where } y(\zetab)\in V \text{ solves   }\\
\langle e(y(\zetab),\zetab),v\rangle = \langle Bu,v\rangle \quad\forall v\in V,\text{ }\rho\text{-a.e. } \zetab \in\Gamma,\\
\end{cases}
\end{equation}
where $U$ and $V$ are Hilbert spaces.
The random vector accounts for the uncertainties in the PDE-model, since coefficients, forcing terms, boundary/initial conditions or shape of the domain may not be completely known either due to a lack of knowledge, measurement errors or intrinsic randomness in the system. $F$ is the convex quantity of interest to minimize. Examples are a tracking term which measures the distance of the state $y(\zetab)$ from a desirable state $y_d$, the flux across a part of the boundary, the maximal deflection of a structure under external loads, or the aerodynamic forces over an airfoil. The random PDE (possibly nonlinear) is represented by $e(\cdot,
\cdot):V\times \Gamma\rightarrow V^\prime$, and $B:U\rightarrow V^\prime$ is a linear bounded operator that describes how the control acts on the state (e.g., distributed or boundary control). 

A standard approach to solve \eqref{eq:OCP_model_intro} is based on a Sample Average Approximation (SAA) \cite{shapiro2021lectures}, which consists in replacing the (exact) expectation with an empirical approximation 
\[\E{F(y(\zetab))}\approx \sum_{n=1}^M w_nF(y(\bar{\zetab}_n)),\]
and consequently to collocate the $\rho$-a.e. PDE-constraint onto the set of points $\left\{\bar{\zetab}_n\right\}_{n=1}^M$. The points are randomly generated in a Monte Carlo approach, or chosen deterministically according to low discrepancy sets in a Quasi-Monte Carlo quadrature \cite{kroese2013handbook}. 
The drawback of both approaches is that their discretization error slowly decays to zero with respect to the number of quadrature points $M$, see \cite{Matthieu,doi:10.1137/19M1294952}, 

If the solution map $\zetab \in \Gamma\rightarrow p(\zetab)\in V$, $p$ being the adjoint variable, is sufficiently  smooth, the Stochastic Collocation Method (SCM) \cite{babuvska2010stochastic} is a valid alternative to Monte Carlo based approximations. The SCM approximates the exact expectation in \eqref{eq:OCP_model_intro} with a quadrature formula obtained as the tensor product of one-dimensional (possibly Gaussian) quadrature operators, 
which may lead to the exponential convergence
\[\|u-u_{\text{SC},\betab}\|_{L^2(D)}\leq C_{SC}\sum_{n=1}^N e^{-\widetilde{g}_n\beta_n},\] where $\beta_n$ is the number of quadrature points used for the $n$-th parameter $\zeta_n$, and 
$\widetilde{{\bm g}}=\left\{\widetilde{g}_n\right\}_{n=1}^N$ are a set of coefficients depending on the region of holomorphy of the map $\zetab\in \Gamma \rightarrow p(\zetab)\in V$ in the complex plane \cite{Matthieu,babuvska2010stochastic}. See also \cite{kunoth2013analytic,kunoth2016sparse} for parametric regularity analyses of optimal control problems with parameter dependent (stochastic) control $u=u(\zetab)$.
Despite the exponential convergence with respect to $\beta_n$, the SCM suffers the curse of dimensionality if tensor product grids are used, since the number of collocation points grows exponentially with the dimension $N$ of $\zetab$, being $M=\prod_{i=1}^N \beta_j$.
Sparse approximations have been extensively studied to approximate the state solution map $\zetab\mapsto y(\zetab)$ (see, e.g., \cite{schwab2011sparse,cohen2015approximation,cohen2010convergence,cohen2011analytic}), and may lead to dimension independent convergence rates in favourable cases. In the context of OCP, sparse grids have been used in \cite{kouri2012approach,kouri2013trust} to reduce the number of samples for large values of $N$. On the other hand, \cite{van2019robust} approximates the expectation in \eqref{eq:OCP_model_intro} using the Multilevel Monte Carlo Method (MLMC), which does not necessarily involve less samples, but does reduce the computational cost by exploiting coarser discretization of the random PDE constraint.
However, either approach is not always fully justified, since they both involve negative quadrature weights, which may destroy the convexity of the original OCP \eqref{eq:OCP_model_intro}. Furthermore, even in the simple linear-quadratic case, negative weights introduce difficulties in developing efficient iterative solvers since the positive definiteness of the reduced optimality system cannot be guaranteed (thus preventing the use of conjugate gradient), while the preconditioners studied in \cite{kouri2018inexact,vanzan} for the full-space optimality system do require positive weights.
 
In this manuscript, we propose a Combination Technique (CT) to solve efficiently \eqref{eq:OCP_model_intro}, inspired by the Multi-Index Stochastic Collocation method (MISC) for elliptic random PDEs in the context of forward Uncertainty Quantification (UQ) \cite{haji2016multi,haji2016multi2}, and by the seminal works concerning the solution of parametric high dimensional PDEs \cite{griebel1990combination,pflaum,combination_te}.
The CT relies on a hierarchical representation of the optimal control specified by a set of multi-indices. Each multi-index corresponds to a level of discretization in the parameters $\left\{\zeta_n\right\}_{n=1}^N$ and in the spatial approximation of the PDE constraint. The CT requires to solve the OCP \eqref{eq:OCP_model_intro} several times, but each instance never uses a tensor product grid with many quadrature points and a fine PDE discretization at the same time. All the optimal controls computed on these coarser tensor product grids and meshes are then recombined to get a final approximation. Under suitable regularity assumptions, the CT can attain the same accuracy of the (full tensor) SCM solution, yet with a highly reduced computational cost.
Thus, the CT is a valid alternative to both sparse grids and MLMC approximations, since it effectively reduces the number of quadrature points for high dimensional stochastic input vectors (as sparse grid approximations do), and further reduces the computational cost by moving most of the computational work on coarser spatial grids (as MLMC does). 
Moreover, the CT solves exclusively OCPs discretized on tensor product grids, thus the quadrature weights are positive and all the discretized OCPs are convex.

Key to the performance of the CT method is the choice of the set of multi-indices to be retained in the CT approximation.
We here propose an a-priori construction of such index set based on a scalar quantity, called profit, assigned to each multi-index, as in \cite{haji2016multi,nobile2016convergence,gerstner2003dimension}. The profit of a multi-index measures how advantageous it is to include it into the set, by taking into account both its error and work contributions. Inherently in its construction, the CT approximation balances the spatial and stochastic error discretization according to this profit metric. 
We further present an analysis which characterizes the asymptotic complexity of the CT, under the assumption of exponential convergence of the tensor grid SCM, algebraic convergence of the PDE spatial discretization and finite dimension of the random vector $\zetab$. We show that the asymptotic complexity of the method depends exclusively on the deterministic solver used for the random PDE, as for the MISC method \cite{haji2016multi} for forward UQ. Note that, while the complexity analysis is restricted to the finite dimensional case, the computational framework presented here can be extented to the infinite dimensional setting following, e.g., \cite{ernst2018convergence,chkifa2014high} (see Remark 3).

Finally, albeit this work focuses on the SCM with Gaussian quadrature, we remark that the CT can be easily modified to accommodate other multilevel/multi-index quadrature formulas \cite{harbrecht2012multilevel} such as multilevel/multi-index (Quasi-)Monte Carlo methods (see Remark 2).

The rest of the manuscript is organized as follows. Section \ref{sec:problem_setting} introduces the class of problems considered.
Section \ref{sec:combination_technique} introduces the CT approximation. Section \ref{sec:construction} discusses the construction of the multi-index set which enters into the definition of the CT approximation. Two possible constructions are detailed: an adaptive procedure based on the dimension adaptive algorithm of \cite{gerstner2003dimension}, and an a-priori construction based on some theoretical assumptions.
A theoretical complexity analysis is further presented for the a-priori construction (details of the proofs are reported in the Appendix).
Section \ref{sec:numerical} discusses the validity of the assumptions for a model problem and presents numerical results that show the effectiveness of the CT.

\section{Problem setting}\label{sec:problem_setting}
Let $\mathcal{D}$ be a Lipschitz bounded domain in $\mathbb{R}^d$ and $\zetab=(\zeta_1,\zeta_2,\cdots,\zeta_N)$ a $N$-dimensional random vector whose components are mutually independent, and uniformly distributed on $\Gamma:=\times_{n=1}^N\Gamma_n$, with probability density $\rho(\zetab)=\prod_{n=1}^N \rho_n(\zeta_n)d\zeta_n$, $\rho_n(\zeta_n)=\frac{1}{|\Gamma_n|}$. Let $\sigma_B(\Gamma)$ be the Borel $\sigma$-algebra over $\Gamma$, and $(\Gamma,\sigma_B(\Gamma),\rho(\zetab)d\zetab)$ the complete probability space with uniform measure. In this manuscript, we consider the random PDE in weak form
\begin{equation}\label{eq:state}
\begin{aligned}
\langle e(y(\zetab),\zetab),v\rangle &= \langle \phi+Bu,v\rangle \quad\forall v\in V,\text{ }\rho\text{-a.e. } \zetab \in\Gamma,\\
\end{aligned}
\end{equation}
where $u\in U$, and $V$ and $U$ are appropriate separable Hilbert spaces.
\begin{assumption}\label{ass:well-posedness}
Problem \eqref{eq:state} is well-posed for $\rho$-a.e. $\zetab\in \Gamma$ and its solution, interpreted as a Hilbert space-valued function $\zetab\in \Gamma \mapsto y(\zetab)\in V$, belongs to the Bochner space 
\[L_{\rho}^p(\Gamma;V):=\left\{u:\Gamma\rightarrow V,\ \text{strongly measurable s.t. } \int_{\Gamma} \|u(\zetab)\|^p_V d\rho(\zetab)<\infty\right\},\]
for some $p\geq 1$. Further, \eqref{eq:state} defines a G\^{a}teaux differentiable operator $S_{\zetab}:U\rightarrow V$ such that $y(\zetab)=S_{\zetab}(u)$ for $\rho$-a.e. $\zetab$.
\end{assumption}

In several applications the interest lies in computing some convex and Fr\'echet differentiable quantity of interest of the random solution $y(\zetab)$, which we denote by $F:L_{\rho}^p(\Gamma;V)\rightarrow L_{\rho}^q(\Gamma;\mathbb{R})$, for some $q\geq 1$. The goal is to optimally control the system, by minimizing the expected value of the quantity of interest, through an external control $u\in U$. An additional penalization on the energy of the control $u$ can be added. In mathematical terms, this translates into the optimal control problem constrained by the random PDE
\begin{equation}\label{eq:OCP}
\begin{cases}
\min_{u\in U} \E{F(y(\zetab))}+\frac{\nu}{2}\|u\|^2_{U}\\
\text{where } y(\zetab)\in V \text{ solves   }\\
\langle e(y(\zetab),\zetab),v\rangle = \langle \phi+ Bu,v\rangle \quad\forall v\in V,\text{ }\rho\text{-a.e. } \zetab \in\Gamma,\\
\end{cases}
\end{equation}
with $\nu\in\mathbb{R}^+$. We denote the reduced objective functional by $J:U\rightarrow \mathbb{R}$, with $J(u):=\E{F(S_{\zetab}(u))}+\frac{\nu}{2}\|u\|_U^2$.
We make the following working assumption. 
\begin{assumption}\label{ass:well-posedness_OCP}
Problem \eqref{eq:OCP} admits a solution $u\in U$ which satisfies the optimality condition
\begin{equation}\label{eq:optimality_condition}
\langle J^\prime(u),v\rangle=\langle \nu u -B^\star \E{ p},v\rangle =0,\quad \forall v\in U,
\end{equation}
where $B^\star$ is the adjoint operator of $B$, and the state and adjoint variables $(y,p)$ satisfy the optimality system
\begin{align*}
\langle e(y(\zetab),\zetab),v\rangle&=\langle \phi+ B B^\star \frac{\E{p}}{\nu},v\rangle \quad\forall v\in V,\text{ }\rho\text{-a.e. } \zetab \in\Gamma,\\
\langle e^\star(p(\zetab),\zetab),v\rangle&=\langle F^\prime(y(\zetab)),v\rangle \quad\forall v\in V,\text{ }\rho\text{-a.e. } \zetab \in\Gamma,
\end{align*}
$F^\prime$ being the Fr\'{e}chet derivative of $F$, and $e^\star$ being the adjoint operator of $e$.
Further, $J$ is twice Fr\'{e}chet differentiable, and its Hessian evaluated in $u$ is positive definite, with bounded inverse, and locally Lipschitz in a neighboorhood of $u$.
\end{assumption}
Sufficient conditions for the validity of Assumption \ref{ass:well-posedness_OCP} can be found in, e.g., \cite{kouri2018optimization,doi:10.1137/16M1086613,martinez2018optimal}.

\begin{example}\label{example:1}
To fix ideas, the reader may consider the elliptic random PDE that has been extensively studied in the context of UQ \cite{lord_powell_shardlow_2014,babuvska2010stochastic,cohen2015approximation},  namely
\begin{equation}\label{eq:state_example}
\begin{aligned}
\langle e(y(\zetab),\zetab),v\rangle=\int_\mathcal{D} \df(x,\zetab)\nabla y(x,\zetab)\cdot \nabla v(x)dx &= \int_\mathcal{D} \phi(x)v(x)dx \quad\forall v\in V,\text{ }\rho\text{-a.e. } \zetab \in\Gamma.\\
\end{aligned}
\end{equation}
where $V=H^1_0(\mathcal{D})$.
If the random diffusion field $\df$ is strictly positive and bounded with probability $1$, i.e. there exists $\df_{\min}>0$ and $\df_{\max}<\infty$ such that 
\begin{equation}
P(\df_{\min}\leq \df(x,\zetab)\leq \df_{\max},\forall x\in \overline{\mathcal{D}})=1,
\end{equation}
then \eqref{eq:state} is well-posed and $y\in L^\infty(\Gamma;V)$.
As an instance of \eqref{eq:OCP}, we may consider the tracking-type problem
\begin{equation}\label{eq:OCP_example}
\begin{cases}
\min_{u\in \Ldue} J(u)=\frac{1}{2}\E{\|y(\zetab)-y_d\|^2_{\Ldue}} +\frac{\nu}{2}\|u\|^2_{\Ldue},\\
\text{where } y(\zetab)\in V \text{ solves   }\\
\int_\mathcal{D} \df(x,\zetab)\nabla y(x,\zetab)\cdot \nabla v(x)dx = \int_\mathcal{D} (\phi(x)+u(x))v(x)dx \quad\forall v\in V,\text{ }\rho\text{-a.e. } \zetab \in\Gamma,
\end{cases}
\end{equation}
for which Assumption \ref{ass:well-posedness_OCP} is verified, see \cite[Section 6]{kouri2018optimization}.
\end{example}

\section{Combination technique}\label{sec:combination_technique}

In this section we present the CT to approximate the solution of \eqref{eq:OCP}. To do so, we introduce two multi-indices $\betab$ and $\alphab$. The first one is associated to the quadrature formula used in the stochastic space, while the second one defines the level of approximation with respect to the physical variables.
 
Let $\betab=(\beta_1,\beta_2,\dots,\beta_N)\in \mathbb{N}_+^N$, $\mathbb{N}_+:=\left\{1,2,\dots\right\}$, be a multi-index, $m:\mathbb{R}^+\rightarrow \mathbb{R}^+$ a strictly increasing function, and $\left\{\mathcal{Q}_j^{\beta_j}\right\}_{j=1}^N$ a set of one-dimensional quadrature operators defined as 
\begin{equation}\label{eq:1D_quadrature}
\mathcal{Q}_j^{\beta_j}[F(y(\zetab))]=\sum_{n_j=1}^{m(\beta_j)} w^{\beta_j}_{n_j} F(y(\zeta_1,\dots,\zeta_{j-1},\bar{\zeta}^{\beta_j}_{n_j},\zeta_{j+1},\dots,\zeta_N))\approx \int_{\Gamma_j} F(y(\zeta_1,\dots,\zeta_N))\rho_j(\zeta_j)d\zeta_{j}.
\end{equation}
where $(\bar{\zeta}^{\beta_j}_{n_j},w^{\beta_j}_{n_j})_{n_j=1}^{m(\beta_j)}$ are the nodes and weights of the $\beta_j$-th quadrature formula. To ensure good approximation properties, the nodes are usually chosen according to the underlying probability measure $\rho_n(\zeta_n)d\zeta_n$. Standard choices for uniform random variables are Gauss-Legendre, Clenshaw-Curtis, or Leja points \cite{reichel1990newton,trefethen2008gauss}. The map $m(\cdot)$ specifies how the number of quadrature nodes depends on $\beta_j$. Common choices are the linear map $m(\beta_j)=\beta_j$, or the doubling map $m(\beta_j)=2^{\beta_j}-1$, see \cite{nobile2016convergence}. The multi-dimensional quadrature operator is directly obtained as the tensorization of the one-dimensional ones,
\begin{equation}\label{eq:tensorized_quadrature_formula}
\EQ{F(y(\zetab))}:=(\mathcal{Q}_1^{\beta_1}\otimes\cdots \otimes \mathcal{Q}_N^{\beta_N})[F(y(\zetab))]= \sum_{n_1=1}^{m(\beta_1)}\dots\sum_{n_N=1}^{m(\beta_N)} F(y(\bar{\zeta}^{\beta_1}_{n_1},\dots,\bar{\zeta}^{\beta_N}_{n_N}))w^{\beta_1}_{n_1}\cdots w^{\beta_N}_{n_N}.
\end{equation}
Denoting with $\Lambda_{\betab}\subset \Gamma$ the set of all quadrature nodes in \eqref{eq:tensorized_quadrature_formula} associated to the multi-index $\betab$, the semi-discretization of \eqref{eq:OCP} is 
\begin{equation}\label{eq:OCP_semidiscrete}
\begin{cases}
\min_{u\in U} \frac{1}{2}\EQ{F(y(\zetab))} +\frac{\nu}{2}\|u\|^2_{U},\\
\text{where } y(\zetab)\in V \text{ solves   }\\
\langle e(y(\zetab),\zetab),v\rangle = \langle \phi+Bu,v\rangle,\quad \forall v\in V,\ \forall \zetab \in \Lambda_{\betab}.
\end{cases}
\end{equation}

Next, let $\alphab=(\alpha_1,\alpha_2,\cdots,\alpha_D)\in \mathbb{N}^D$ be a second multi-index. In general, $D$ does not necessarily coincide with the number of physical dimensions. The multi-index $\alphab$ can tune the level of discretization along each physical dimension for meshes whose level of refinement can be set independently for each dimension (e.g., meshes obtained as the tensor product of one-dimensional meshes, structured meshes with a regular connectivity, or more general domains discretized using isogeometric analysis, see \cite{beck2019iga}). Otherwise, $\alphab$ can be a simple scalar which tunes the reference mesh-size across the whole domain $\mathcal{D}$, as in classical multilevel methods \cite{harbrecht2012multilevel,teckentrup2015multilevel,van2019robust}.  In a broader picture, $\alphab$ could even represent time-steps or describe a general set of multi-fidelity models \cite{haji2016multi2,peherstorfer2018survey}. 
Nevertheless, for the sake of this manuscript, we suppose that the $j$-th component of $\alphab$ determines the mesh size $h_{j,\alpha_j}$ along the $j$-th physical dimension, and we denote with $U^{\alphab}$ and $V^{\alphab}$ the finite-dimensional approximations (e.g, finite element approximations on a mesh specified by $\alphab$) of $U$ and $V$. 

Now, let $u^{\alphab,\betab}$ be a solution of the fully discrete OCP
\begin{equation}\label{eq:OCP_definition}
\begin{cases}
\min_{u\in U^{\alphab}} \frac{1}{2}\EQ{F(y(\zetab))} +\frac{\nu}{2}\|u\|^2_{U},\\
\text{where } y(\zetab)\in V^{\alphab} \text{ solves   }\\
\langle e(y(\zetab),\zetab),v\rangle = \langle \phi+Bu,v\rangle,\quad \forall v\in V^{\alphab},\ \forall \zetab \in \Lambda_{\betab},
\end{cases}
\end{equation}
which can be computed with any suitable optimization algorithm, see, e.g, \cite{hinze2008optimization}, \eqref{eq:OCP_definition} being a finite-dimensional optimization problem.
We assume that $u^{\alphab,\betab}\xrightarrow{U} u$ as $(\alphab,\betab)\rightarrow \infty$ (component-wise). Using, e.g., consistent finite element and stochastic collocation discretizations on tensor product grids, this is verified in the linear-quadratic case due the global convexity of the discretized objective functional, see, e.g., \cite{doi:10.1137/19M1294952}. For a general formulation such as \eqref{eq:OCP}, which may involve local minima, sufficient conditions are provided by the Brezzi-Rappaz-Raviart theory \cite{brezzi1980finite} applied to the first order optimality conditions, namely the discretized gradient evaluated in $u$ converges in norm to zero,  the discretized objective functional is twice differentiable, whose Hessian evaluated in $u$ is positive definite, with a uniformly (with respect to $(\alphab,\betab)$) bounded inverse, and locally Lipschitz continuous in a neighboorhod of $u$. 

Next, we introduce the first order difference operators for the deterministic and stochastic discretization parameters
\begin{align}
&\Delta_i^{\dete}[u^{\alphab,\betab}]=\begin{cases}
u^{\alphab,\betab}-u^{\alphab-\eb{i},\betab},\quad &\text{if }\alpha_i>1,\\
u^{\alphab,\betab},\quad &\text{if }\alpha_i=1,
\end{cases}\\
&\Delta_i^{\stoc}[u^{\alphab,\betab}]=\begin{cases}
u^{\alphab,\betab}-u^{\alphab,\betab-\eb{i}},\quad &\text{if }\beta_i>1,\\
u^{\alphab,\betab},\quad &\text{if }\beta_i=1,
\end{cases}
\end{align}
where $\eb{i}$ is the i-th canonical vector.
The deterministic and stochastic hierarchical surpluses are defined as the tensorized product of their respective first-order differences,
\begin{align}
\Delta^{\dete}[u^{\alphab,\betab}]&=\otimes_{i=1}^{D} \Delta_i^{\dete}[u^{\alphab,\betab}]=\Delta_1^{\dete}\bigg[\Delta_2^{\dete}\bigg[\dots \Delta_D^{\dete} [u^{\alphab,\betab}]\bigg]\bigg]=\sum_{\jb\in \left\{0,1\right\}^D} (-1)^{|\jb|}u^{\alphab-\jb,\betab},\label{eq:deterministic_surplus}\\
\Delta^{\stoc}[u^{\alphab,\betab}]&=\otimes_{i=1}^{N} \Delta_i^{\stoc}[u^{\alphab,\betab}]=\Delta_1^{\stoc}\bigg[\Delta_2^{\stoc}\bigg[\dots \Delta_N^{\stoc} [u^{\alphab,\betab}]\bigg]\bigg]=\sum_{\jb\in \left\{0,1\right\}^N} (-1)^{|\jb|}u^{\alphab,\betab-\jb},\label{eq:stochastic_surplus}
\end{align}
where we used the standard notation $u^{\alphab,\betab}=0$ if any of the $\alpha_j$ or $\beta_j$ is zero. Moreover $|\cdot|$ denotes the $l_1$ norm of the vector (in this case, the number of nonzero components).

Given a multi-index set $\I\subset \mathbb{N}_+^{D+N}$, the CT approximation $\mathcal{M}_{\I}(u)$ for the optimal control $u$ solution of \eqref{eq:OCP} is defined as the truncated telescopic sum,
\begin{align}\label{eq:CT_estimator}
u=\sum_{(\alphab,\betab)\in \mathbb{N}^{N+D}_{+}} \Deltab\big[ u^{\alphab,\betab}\big]\approx \sum_{(\alphab,\betab)\in \I} \Deltab\big[ u^{\alphab,\betab}\big]=: \mathcal{M}_{\I}(u),
\end{align}
where $\Deltab\big[ u^{\alphab,\betab}\big]:=\Delta^{\stoc}\big[ \Delta^{\dete}\big[u^{\alphab,\betab}\big]\big]$. In this work we \textit{assume} that the infinite series converges absolutely in $U$, namely that $u^{\alphab,\betab} \xrightarrow{U} u$ as $(\alphab,\betab)\rightarrow \infty$ (component-wise) and that $\left\{\| \Delta u^{\alphab,\betab}\|_{U}\right\}_{(\alphab,\betab)\in \mathbb{N}_+^{N+D}}\in \ell^1(\mathbb{N}_+^{N+D})$. In the rest of the manuscript, we assume the existence of specific bounds in Assumption \ref{ass:error}, and carry out a complexity analysis under these hypotheses in Section \ref{sec:complexity}. Their validity is checked numerically on an example in Section \ref{sec:numerical}.

\begin{remark}[On the convergence of the infinite series]
Sufficient conditions for the convergence of the minimizers of the discretized OCPs to $u$ have been previously recalled. 
On the other hand, the summability of the norms of the hierarchical surpluses may require additional hypotheses, and its
analysis involves the cumbersome analytical derivation of suitable bounds. This is the subject of current endeavours in the linear-quadratic case. We remark that the validity of similar hierarchical expansions has been demonstrated in several context and it relies on the mix-regularity with respect to the spatial and stochastic variables \cite{griebel1990combination,combination_te,MIMC,nobile2016convergence,haji2016multi2}.
\end{remark}

As in sparse grids \cite{bungartz2004sparse}, the rational behind the truncation is that not all hierarchical surpluses $\Deltab\big[ u^{\alphab,\betab}\big]$ have the same relevance. Therefore, under suitable  assumptions (discussed in Section \ref{sec:apriori}), it is possible to retain only a few hierarchical surpluses while preserving accuracy at a reduced computational cost. The CT approximation \eqref{eq:CT_estimator} admits the following equivalent formulation \cite{bungartz2004sparse}
\begin{align}\label{eq:CT_estimator_equiv}
\mathcal{M}_{\I}(u)=\sum_{(\alphab,\betab)\in \I} c_{\alphab,\betab} u^{\alphab,\betab},\qquad c_{\alphab,\betab}=\sum_{\jb\in \left\{0,1\right\}^{N+D}: (\alphab,\betab)+\jb \in \I} (-1)^{|\jb|}
\end{align}
which motivates the name ``combination technique" since $\mathcal{M}_{\I}(u)$ is obtained as a combination of the solution of \eqref{eq:OCP_definition} for different multi-indices $(\alphab,\betab)$.  Further, the coefficients $c_{\alphab,\betab}$ are zero whenever $(\alphab,\betab)+\jb\in \I$, $\forall \jb\in\left\{0,1\right\}^{N+D}$. Hence, only a few multi-indices $(\alphab,\betab)\in \I$ actually contribute to the CT approximation.

\begin{remark}[Sparse grids vs Combination technique]\label{remark:sparse_grids}
Sparse grids approximations of OCPs under uncertainty \cite{kouri2012approach,kouri2013trust} rely on first order differences of the one-dimensional quadrature operators
\[\widetilde{\Delta}_n\LQ \mathcal{Q}_n^{\beta_n}\RQ := \mathcal{Q}_n^{\beta_n}-\mathcal{Q}_n^{\beta_n-1},\quad n=1,\dots,N,\] 
and on the related hierarchical surpluses $\widetilde{\Delta}\LQ\mathcal{Q}^{\betab}\RQ:=\left(\widetilde{\Delta}_1\otimes\cdots\otimes  \widetilde{\Delta}_N\right)\LQ \mathcal{Q}^{\betab}\RQ$.
The sparse grid quadrature is then $\widetilde{\mathbb{E}}_{\mathcal{I}}=\sum_{\betab\in \mathcal{I}} \widetilde{\Delta}\LQ\mathcal{Q}^{\betab}\RQ$ and
the sparse grid semi-discrete OCP reads
\begin{equation}\label{eq:OCP_sparsegrids}
\begin{cases}
\min_{u\in U} \frac{1}{2}\widetilde{\mathbb{E}}_{\mathcal{I}}\LQ F(y(\zetab))\RQ +\frac{\nu}{2}\|u\|^2_{U},\\
\text{where } y(\zetab)\in V \text{ solves   }\\
\langle e(y(\zetab),\zetab),v\rangle = \langle \phi+u,v\rangle,\quad \forall v\in V,\ \forall \zetab \in \widetilde{\Lambda}_{\betab},
\end{cases}
\end{equation}
where $\widetilde{\Lambda}_{\betab}$ collects all the quadrature points involved by the sparse grid quadrature. It is well-known that the weights of sparse grids quadrature formulae can be negative. This may break the convexity of the OCP.
In contrast, if we fix a mesh discretization, the CT replaces the sparse approximation of the expectation operator with a hierarchical representation of $u$, using first order differences of the solution $u^{\betab}$  of the OCP computed with the tensor product quadrature determined by $\betab$. Hence, each OCP involves only quadratures with positive weights and the (reduced on $u$) discrete OCPs remain convex.
\end{remark}

\begin{remark}[Multilevel/multi-index (Quasi-)Monte Carlo methods]
The CT framework covers as particular instances other multilevel quadrature rules \cite{harbrecht2012multilevel}. For instance, let $\alpha$ be a scalar ($D=1$) setting the overall mesh size. Then, given a finest approximation level $L$, we may consider the telescopic sum
\[u\approx u^L=\sum_{\alpha=1}^{L} \Delta_1^{\det}\LQ u^{\alpha}\RQ = \frac{1}{\nu}B^\star\sum_{\alpha=1}^{L} \E{p^{\alpha}-p^{\alpha-1}}.\]
The expectation of the difference $p^{\alpha}-p^{\alpha-1}$ is now prone to be approximated with a level-dependent (Quasi-)Monte Carlo quadrature formula obtaining a multilevel (Quasi-)Monte Carlo method to solve  
\eqref{eq:OCP}, which again preserve the (possible) convexity by avoiding negative quadrature weights, since $p^{\alpha}$ and $p^{\alpha-1}$ are the adjoint variables of two distinct OCPs discretized with the same (Quasi-)Monte Carlo samples but on different spatial meshes.
Notice that if $\alphab\in \mathbb{N}^D$, $D>1$, the CT can accommodate the multi-index (Quasi-)Monte Carlo method \cite{MIMC,doi:10.1137/16M1082561}.
\end{remark}

\section{Construction of the multi-index set $\I$}\label{sec:construction}
It is clear that the multi-index set $\I$ determines the computational efficiency of the CT approximation. Heuristically, $\I$ should contain very few multi-indices that simultaneously lead to fine physical discretizations and to quadrature formulae with a large number of quadrature nodes, since they would require the solution of very expensive OCPs. In contrast, $\I$ should contain ``sparse" multi-indices which lead to fine discretizations/high level quadrature only across very few spatial and stochastic dimensions, while having coarse spatial/stochastic discretizations for most of the variables. 
  
A well-known strategy to find a (quasi-)optimal multi-index set $\I$ is to recast its construction as a knapsack optimization problem, see, e.g., \cite{gerstner2003dimension,haji2016multi,nobile2016convergence}.
To do so, we introduce the concepts of work contribution and error contribution associated to the hierarchical surplus $\Deltab\big[u^{\alphab,\betab}\big]$. 
The work contribution $\Deltab W_{\alphab,\betab}$ measures the computational cost required to add $\Deltab\big [u^{\alphab,\betab}\big]$ to $\mathcal{M}_{\I}(u)$. 
In formulae we set
\[\Deltab W_{\alphab,\betab}=\text{Work}\big[ \mathcal{M}_{\I\cup (\alphab,\betab)}(u)\big]-\text{Work}\big[ \mathcal{M}_{\I}(u)\big]=\text{Work}\big[ \Deltab \big[ u^{\alphab,\betab}\big]\big],\]
which implies
\[\text{Work}\big[\mathcal{M}_{\I}(u)\big]=\sum_{(\alphab,\betab)\in\I}\Deltab W_{\alphab,\betab}.\]
The error contribution $\Delta E_{\alphab,\betab}$ measures instead how much the error $u-\mathcal{M}_{\I}(u)$ varies if $\Deltab\big [u^{\alphab,\betab}\big]$
is added to the estimator $\mathcal{M}_{\I}$,
\[\Deltab E_{\alphab,\betab}=\|u-\mathcal{M}_{\I\cup (\alphab,\betab)}-u+\mathcal{M}_{\I}\|_{U}=\|\mathcal{M}_{\I\cup(\alphab,\betab)}(u)-\mathcal{M}_{\I}(u)\|_{U}=\|\Deltab [u^{\alphab,\betab}]\|_{U}.\]
It follows that the error of the CT approximation is bounded by the sum of error contributions not included in the set $\I$, that is,
\begin{align}\label{eq:error_bound}
\text{Error}\big[\mathcal{M}_{\I}(u)\big]=\|u-\mathcal{M}_{\I}(u)\|_{U}&=\left\|\sum_{(\alphab,\betab)\notin \I} \Deltab\big[u^{\alphab,\betab}\big]\right\|_{U}\leq \sum_{(\alphab,\betab)\notin \I} \Deltab E_{\alphab,\betab},
\end{align}
provided that the series converges.
To construct $\I$, we wish to solve a binary knapsack problem, i.e. to maximize the sum of error contributions included in $\I$ subject to a maximum work $W_{\max}$ available,
\begin{equation}\label{eq:knapsack_problem}
\begin{aligned}
\max\limits_{\chi_{\alphab,\betab}\in \left\{0,1\right\}}&  \sum_{(\alphab,\betab)\in \mathbb{N}_+^{D+N}} \Delta E_{\alphab,\betab} \chi_{\alphab,\betab},\\
\text{s.t. }& \sum_{(\alphab,\betab)\in \mathbb{N}_+^{D+N}} \Delta W_{\alphab,\betab} \chi_{\alphab,\betab}\leq W_{\max}.
\end{aligned}
\end{equation}
Problem \eqref{eq:knapsack_problem} leads to a \textit{quasi-optimal} multi-index set $\I$, as the sum of the errors contributions is only an upper bound of the actual error of the CT approximation. Moreover, the solution of \eqref{eq:knapsack_problem} is computationally unfeasible. A valid alternative is to relax the integer constraint on $\chi_{\alphab,\betab}$, and solve the relaxed problem using the Dantzig algorithm \cite{martello1990knapsack}, which requires to
\begin{itemize}
\item[1] For each hierarchical surplus, compute the profit $P_{\alphab,\betab}=\frac{\Delta E_{\alphab,\betab}}{\Delta W_{\alphab,\betab}}$.
\item[2] Sort the hierarchical surpluses by decreasing profit.
\item [3] Add the hierarchical surpluses to the multi-index set $\I$ in such order, until the work constraint is satisfied.
\end{itemize}

Given a tolerance $\varepsilon$, the quasi-optimal multi-index set is thus given by 
\begin{equation}\label{eq:multiindex_set}
\I(\varepsilon)=\left\{(\alphab,\betab)\in \mathbb{N}_+^{N+D}:\ P_{\alphab,\betab}\geq \varepsilon\right\}.
\end{equation}

Notice that the above discussion assumes that the error and work contributions are known, but this is not generally the case.
To build $\I$ in practice, a possible solution is an adaptive procedure, as the one proposed in \cite{gerstner2003dimension}. The general algorithm is described in Alg. \ref{alg:1}. 
\begin{algorithm}[]
\setlength{\columnwidth}{\linewidth}
\caption{Adaptive incremental construction of $\I$.}
\begin{algorithmic}[1]\label{alg:1}
\REQUIRE Tolerance $\varepsilon>0$. \\
\STATE $\mathbf{v}=(1,\dots,1)\in \mathbb{N}_+^{D+N}$.\\
\STATE $\mathcal{E}=\left\{\mathbf{v}\right\}$, $\I=\emptyset$, $\mathcal{M}_{\I}\LQ u\RQ=0,$ $\text{Err}=0$. \\
\STATE Compute $u^{\mathbf{v}}$.
\STATE Compute $\Delta\LQ u^{\mathbf{v}}\RQ$, $\Delta E_{\mathbf{v}}$, $\Delta W_{\mathbf{v}}$ and $P_{\mathbf{v}}$ numerically.\\
\STATE Set $\text{Err}=\text{Err}+\Delta E_{\mathbf{v}}$.
\WHILE{$\text{Err}>\varepsilon$}
\STATE Select multi-index $\mathbf{v}=\text{argmax}_{\bm{\ell}\in \mathcal{E}} P_{\bm{\ell}}$. \\
\STATE Set $\mathcal{M}_{\I}\LQ u\RQ =\mathcal{M}_{\I}\LQ u\RQ + \Delta\LQ u^{\mathbf{v}}\RQ$.
\STATE Set $\I=\I\cup \mathbf{v}$, $\mathcal{E}=\mathcal{E}\setminus \mathbf{v}$ and $\text{Err}=\text{Err}-\Delta E_{\mathbf{v}}$.\\
\FOR{$k=1,\dots ,D+N$}
\STATE $\mathbf{l}=\mathbf{v}+\mathbf{e}_k$.\\
\IF{$\mathbf{l}-\mathbf{e}_i\in \I$ for every $i=1,\dots D+N$}
\STATE $\mathcal{E}=\mathcal{E}\cup \mathbf{l}$.\\ 
\STATE Compute $u^{\mathbf{l}}$.
\STATE Compute $\Delta\LQ u^{\mathbf{l}}\RQ$, $\Delta E_{\mathbf{l}}$, $\Delta W_{\mathbf{l}}$ and $P_{\mathbf{l}}$ numerically.\\
\STATE Set $\text{Err}=\text{Err}+\Delta E_{\mathbf{l}}$.
\ENDIF
\ENDFOR
\ENDWHILE
\STATE Output: $\I$, $\mathcal{M}_{\I}\LQ u\RQ$ and $\text{Err}$.
\end{algorithmic}
\end{algorithm}

This iterative algorithm computes numerically the error contributions $\Delta E_{\alphab,\betab}=\|\Delta u^{\alphab,\betab}\|_{U}$ for every multi-index $(\alphab,\betab)$ in the reduced margin of a current index set, where the reduced margin of a set $\I$ of multi-indices in $\mathbb{N}_+^n$ is
\[\mathcal{R}_\I:=\left\{\mathbf{v}\in \mathbb{N}_+^n:\ \mathbf{v}-\eb{k}\in \I,\text{ for all }k\in \left\{1,\dots,n\right\}, \text{ with }v_k>1\right\}.\]
The algorithm then updates $\I=\I\cup (\bar{\alphab},\bar{\betab})$, where $(\bar{\alphab},\bar{\betab})$ is the multi-index in $\mathcal{R}_\I$ associated to the largest profit. The adaptive construction ends when the sum of the error contributions in the reduced margin is below a desired tolerance, mimicking the error upper bound \eqref{eq:error_bound}.

Despite its simplicity this adaptive algorithm may be not efficient. The cost of computing the CT approximation may be dominated by the cost of constructing the index set $\I$, since evaluating numerically an error contribution (line 15 of Alg. \ref{alg:1}) for a given multi-index in the reduced margin requires to solve an OCP (line 14). Notice that once $u^{\mathbf{v}}$ has been computed, $\Delta u^{\mathbf{v}}$ requires only a linear combination of previously computed solutions.
A possible alternative is to use an a-priori ansatz for $\Delta E_{\alphab,\betab}$ and $\Delta W_{\alphab,\betab}$, so that the computation of the profits in the margin is reduced to the evaluation of a simple algebraic formula, and an OCP is solved exclusively when a selected multi-index is added to $\I$. This variant is described by Alg. \ref{alg:2}.
We will denote the construction of $\I$ based on these adaptive/a-priori frameworks as ``incremental CT". 

\begin{algorithm}[]
\setlength{\columnwidth}{\linewidth}
\caption{A-priori incremental construction of $\I$.}
\begin{algorithmic}[1]\label{alg:2}
\REQUIRE Tolerance $\varepsilon>0$. \\
\STATE $\mathbf{v}=(1,\dots,1)\in \mathbb{N}_+^{D+N}$.\\
\STATE $\mathcal{E}=\left\{\mathbf{v}\right\}$, $\I=\emptyset$, $\mathcal{M}_{\I}\LQ u\RQ=0,$ $\text{Err}=0$. \\
\STATE Compute $\Delta E_{\mathbf{v}}$, $\Delta W_{\mathbf{v}}$ and $P_{\mathbf{v}}$ using ansatzes.\\
\STATE Set $\text{Err}=\text{Err}+\Delta E_{\mathbf{v}}$.
\WHILE{$\text{Err}>\varepsilon$}
\STATE Select multi-index $\mathbf{v}=\text{argmax}_{\bm{\ell}\in \mathcal{E}} P_{\bm{\ell}}$. \\
\STATE Compute $\mathbf{u}^{\mathbf{v}}$ and $\Delta \LQ u^{\mathbf{v}}\RQ$.
\STATE Set $\mathcal{M}_{\I}\LQ u\RQ =\mathcal{M}_{\I}\LQ u\RQ + \Delta\LQ u^{\mathbf{v}}\RQ$.
\STATE $\I=\I\cup \mathbf{v}$ and $\mathcal{E}=\mathcal{E}\setminus \mathbf{v}$ and $\text{Err}=\text{Err}-\Delta E_{\mathbf{v}}$.\\
\FOR{$k=1,\dots D+N$}
\STATE $\mathbf{l}=\mathbf{v}+\mathbf{e}_k$.\\
\IF{$\mathbf{l}-\mathbf{e}_i\in \I$ for every $i=1,\dots D+N$}
\STATE $\mathcal{E}=\mathcal{E}\cup \mathbf{l}$.\\ 
\STATE Compute $\Delta E_{\mathbf{l}}$, $\Delta W_{\mathbf{l}}$ and $P_{\mathbf{l}}$ using ansatzes.\\
\STATE Set $\text{Err}=\text{Err}+\Delta E_{\mathbf{l}}$.
\ENDIF
\ENDFOR
\ENDWHILE
\STATE Output: $\I$ and $\mathcal{M}_{\I}\LQ u\RQ$.
\end{algorithmic}
\end{algorithm}

We remark that a-priori ansatzes on $\Delta E_{\alphab,\betab}$ and $\Delta W_{\alphab,\betab}$ permit to build the multi-index $\I$ set beforehand using formula \eqref{eq:multiindex_set}. In this way, one could use the CT formula \eqref{eq:CT_estimator_equiv}, and reduce the number of OCPs that have to be solved, by considering only the multi-indices such that $c_{\alphab,\betab}\neq 0$. This approach is more efficient but less flexible, since it is inherently non-adaptive. If $\mathcal{M}_{\I(\varepsilon)}[u]$ has been computed, it may not be possible to evaluate $\mathcal{M}_{\I(\varepsilon^\prime)}[u]$, with $\varepsilon^\prime >\varepsilon$, as 
$\mathcal{M}_{\I(\varepsilon^\prime)}[u]$ may involve multi-indices that were inactive (i.e. $c_{\alphab,\betab}= 0$) for $\mathcal{M}_{\I(\varepsilon)}[u]$.
Further, it does not have an intrinsic stopping criterium contrary to the adaptive/incremental CT procedure.
One could also use a-posteriori error estimators for $\Delta E_{\alphab,\betab}$, generalizing the work of \cite{guignard2018posteriori} for random PDEs. However, the derivation of suitable a-posteriori error estimators for OCP constrained by random partial differential equations combined with the CT is still an open problem.

In the next subsection, we assume explicit analytic expressions for the error and work contributions and discuss an a-priori construction of the quasi-optimal set $\I$. These ansatzes contain few parameters that are problem-dependent, and have to be estimated case by case. A complexity analysis of the CT approximation based on this proposed a-priori construction is presented in subsection \ref{sec:complexity}.
The soundness of these expressions have been verified theoretically in the context of random PDEs under suitable regularity assumptions, see \cite{haji2016multi2}. We will verify numerically their validity on a model problem in Section \ref{sec:validation_of_hypothesis}.

\subsection{An a-priori construction}\label{sec:apriori}

To build the set $\I$ in a-priori fashion, we need some hypothesis on the decay of the error and work contributions.
Following the framework proposed in \cite{haji2016multi} for the MISC method, we make the following assumptions.
\begin{assumption}\label{ass:mesh_dependence_linear_map}
The discretization parameters $h_{n,\alpha_n}$ depend exponentially on the discretization level $\alpha_n$, while the number of collocation points grows linearly with the stochastic level $\beta_n$,
\begin{equation}
h_{n,\alpha_n}=h_02^{-\alpha_n},\quad\text{and}\quad m(\beta_n)=\beta_n.
\end{equation}
\end{assumption}
Notice that we assume that the level-to-nodes relation $m(\beta_n)$ is linear, in contrast to several works concerning random PDEs, see, e.g., \cite{haji2016multi,nobile2016convergence}, where the doubling map $m(\beta_n)=2^{\beta_n}-1$ is considered. 
The doubling map is often used together with Clenshaw-Curtis nodes, since the quadrature/interpolation nodes are then nested, which permits to recycle several pre-calculated solutions when computing the quadrature/interpolation for a higher level $\betab$.
In our setting, each multi-index $\betab$ requires the solution of a robust OCP, and none of the previously computed optimal controls (neither state or adjoint variables) can be recycled.
Numerically, we have observed that the linear and doubling mapping lead to the same convergence rate, but the former has a better constant.

In addition to $\Delta W_{\alphab,\betab}$ and $\Delta E_{\alphab,\betab}$ associated to the mixed hierarchical surpluses, for a fixed $\bar{\betab}\in \mathbb{N}_{+}^{N}$, we denote with $\Delta W^{\dete}_{\alphab,\bar{\betab}}$ and $\Delta E^{\dete}_{\alphab,\bar{\betab}}$ the work and error contributions associated to the deterministic hierarchical surplus \eqref{eq:deterministic_surplus}. 
Similarly, for a fixed $\bar{\alphab}\in \mathbb{N}_+^D$, we denote with 
$\Delta W^{\stoc}_{\bar{\alphab},\betab}$ and $\Delta E^{\stoc}_{\bar{\alphab},\betab}$ the work and error contributions associated to the deterministic hierarchical surplus \eqref{eq:stochastic_surplus}.

\begin{assumption}[Assumption on the work contributions]\label{ass:work}
The work contributions associated to the deterministic, stochastic and mixed hierarchical surpluses satisfy\footnote{The work bound involves $\prod_{n=1}^N (\beta_n+1)$ instead of the more natural $\prod_{n=1}^N \beta_n$ to simplify the proofs of Theorems \ref{thm:stoc} and \ref{thm:stoc_det}. Nevertheless, this does not influence the asymptotic analysis since $\prod_{n=1}^N \beta_n\leq \prod_{n=1}^N (\beta_n+1)\leq 2^N\prod_{n=1}^N \beta_n$}
\begin{align}
\Delta W^{\dete}_{\alphab,\bar{\betab}} &\leq C^{\dete,\bar{\betab}}_{\work} \prod_{n=1}^D 2^{\alpha_n\widetilde{\gamma}_n},\label{eq:ass:1}\\
\Delta W^{\stoc}_{\bar{\alphab},\betab} &\leq C^{\stoc,\bar{\alphab}}_{\work} \prod_{n=1}^N (\beta_n+1),\label{eq:ass:3}\\
\Deltab W_{\alphab,\betab}&\leq C_{\work}\left(\prod_{n=1}^D 2^{\alpha_n\widetilde{\gamma}_n}\right)  \left(\prod_{n=1}^N (\beta_n+1)\right),\label{ass:produc_structure_W}
\end{align}
for some rates $\widetilde{\gamma}_n$. 
\end{assumption}
\begin{assumption}[Assumption on the error contributions]\label{ass:error}
The error contributions associated to the deterministic, stochastic and mixed hierarchical surpluses satisfy
\begin{align}
\Delta E^{\dete}_{\alphab,\bar{\betab}} &\leq C^{\dete,\bar{\betab}}_{\error} \prod_{n=1}^D 2^{-\alpha_n \widetilde{r}_n},\label{eq:ass:2}\\
\Delta E^{\stoc}_{\bar{\alphab},\betab} &\leq C^{\stoc,\bar{\alphab}}_{\error} \prod_{n=1}^N e^{-\widetilde{g}_n(\beta_n+1)},\label{eq:ass:4}\\
\Deltab E_{\alphab,\betab}&\leq C_{\error}  \left(\prod_{n=1}^D 2^{-\alpha_n\widetilde{r}_n}\right) \left(\prod_{n=1}^N e^{-\widetilde{g}_n(\beta_n+1)}\right),\label{ass:produc_structure_E}
\end{align}
for some rates $\widetilde{r}_i$ and $\widetilde{g}_i$. 
\end{assumption}
Assumptions \ref{ass:work} and \ref{ass:error} specify how the work and error contributions of the deterministic and stochastic hierarchical surpluses depend on the multi-indices $\alphab$ and $\betab$.
They assume a multiplicative structure of the work and error contributions of the mixed hierarchical surplus. Notice that Assumptions \ref{ass:error} guarantee that $\left\{\Delta E_{\alphab,\betab}\right\}_{(\alphab,\betab)\in \mathbb{N}_+^{N+D}}\in l^1(\mathbb{N}_+^{N+D})$, so that, together with the convergence of the minimizers of the discrete OCPs to $u$, validates the hierarchical representation \eqref{eq:CT_estimator}. The discussion of the rational behind these assumptions is postponed to Section \ref{sec:validation_of_hypothesis}, where numerical evidence of their validity together with a procedure to estimate the rates $\widetilde{r}_i$ and $\widetilde{g}_i$ are presented.

\subsection{Complexity analysis}\label{sec:complexity}
In this subsection, we present a complexity analysis of the CT in two different settings. We first consider the simpler case in which the spatial discretization is fixed, i.e. the multi-index $\alphab$ is an assigned, constant, multi-index  $\alphab=\bar{\alphab}$. In other words, the CT combines the solutions of \eqref{eq:OCP_semidiscrete} for different quadrature formulae, but all solutions are computed on the same spatial mesh. Hence, the CT involves exclusively the stochastic hierarchical surpluses,
\begin{equation}\label{eq:CT_onlystoch}
\mathcal{M}_{\I}(u)=\sum_{\betab\in \I} \Delta^\stoc[ u^{\bar{\alphab},\betab}].
\end{equation}

In the second setting, we consider the more general case where both $\alphab$ and $\betab$ are variables in the CT expansion, as in \eqref{eq:CT_estimator}. In both cases, the complexity analysis is based on a ``direct counting" argument that consists in summing the error contributions outside the index set $\I$, and the work contributions inside $\I$, see \cite{haji2016multi,nobile2016convergence}.

Let us start with the first setting. 
As $\bar{\alphab}$ is fixed, to each multi-index $\betab$ we associate the profit
\[\text{P}_{\betab}=\frac{\Delta E_{\bar{\alphab},\betab}^{\stoc}}{\Delta W_{\bar{\alphab},\betab}^\stoc},\]
and given a tolerance $\varepsilon>0$, the index set $\I$ simplifies to
\[ \I(\varepsilon):=\left\{\betab \in \mathbb{N}_{+}^{N}: \frac{\Delta E_{\bar{\alphab}\betab}^\stoc}{\Delta W_{\bar{\alphab},\betab}^\stoc}\geq \varepsilon\right\}.\]
Due to Assumptions \ref{ass:work} and \ref{ass:error}, $\I(\varepsilon)$ coincides with
\[ \I(L):=\left\{\betab \in \mathbb{N}_{+}^{N}: \widetilde{\gb}\cdot (\betab+\unob) + |\log(\betab+\unob)|_1\leq L\right\},\]
where $|\log(\betab+\unob)|_1=\sum_{i=1}^N \log(\betab_i+1)$, and $L:=L(\varepsilon)=-\log(\varepsilon \frac{C_{\work,\bar{\alphab}}^\stoc}{C_{\error,\bar{\alphab}}^\stoc}).$

The next theorem provides a complexity result for the CT applied only on the stochastic discretization multi-index $\betab$. The proof is detailed in Appendix \ref{appendix:thm:stoc}
\begin{theorem}\label{thm:stoc}
There exist constants $C_1$ and $C_2$ such that for any $W_{\max}$ satisfying $W_{\max}\geq \frac{|\widetilde{\gb}|^{2N}C_1}{(2N)!}$, and choosing $\widehat{L}=\sqrt[2N]{\frac{W_{\max}(2N)!}{C_1}}-|\widetilde{\gb}|$,
\begin{align*}
\text{Work}\big[\mathcal{M}_{\I(\widehat{L})}(u)\big]&\leq W_{\max},\\
\text{Error}\big[\mathcal{M}_{\I(\widehat{L})}(u)\big]&\leq C_2 e^{-\sqrt[2N]{\frac{W_{\max}(2N)!}{C_1}}}\left(\sqrt[2N]{\frac{W_{\max}(2N)!}{C_1}}+1\right)^{2N-1}.
\end{align*}
\end{theorem}
Numerically, we observed that the simplified estimate, based on Stirling approximation,
\begin{equation}\label{eq:decay_simplified}
\text{Error}\big[\mathcal{M}_{\I(\widehat{L})}(u)\big]\leq C e^{-\gamma N\sqrt[2N]{W_{\max}}},
\end{equation} for $C,\gamma\in \mathbb{R}^+$ captures well the convergence (see Fig \ref{Fig:set}).
Notice that the complexity result expressed in \eqref{eq:decay_simplified} is the same found in \cite{nobile2016convergence} for the complexity of sparse grid interpolants (on non-nested grids) for elliptic random PDEs. 

We next consider the general case in which the combination technique is applied both to the spatial and stochastic discretization parameters. Under Assumptions \ref{ass:work}, \ref{ass:error}, and defining  $\rb=\log(2)\widetilde{\rb}$, $\gammab=\log(2)\widetilde{\gammab}$, for a given tolerance $\varepsilon$ the quasi-optimal set \eqref{eq:multiindex_set} becomes
\[ \I(L):=\left\{(\alphab,\betab)\in \mathbb{N}_{+}^{D+N}:\ (\rb+\gammab)\cdot \alphab+\widetilde{\gb}\cdot (\betab+\unob) +|\log(\betab+\unob)|_1\leq L\right\},\]
with $L=-\log\left(\varepsilon\frac{C_{\work}}{C_{\error}}\right)$.
The next Theorem provides an asymptotic convergence result for the CT applied to the OCP \eqref{eq:OCP_definition}.
The proof is detailed in Appendix \ref{appendix:thm:stoc_deter}.
\begin{theorem}\label{thm:stoc_det}
Let $r_j=\log(2)\widetilde{r}_j$, $\gamma_j=\log(2)\widetilde{\gamma}_j$, $j=1,\dots,N$, $\mathbf{\Theta}=(\frac{\gamma_1}{r_1+\gamma_1},\dots,\frac{\gamma_D}{r_D+\gamma_D})$, $\mu=\min_n\frac{r_n}{\gamma_n}$, $\chi=\max_n \mathbf{\Theta}_n$, and $n(\mathbf{\Theta},\chi):=\#\left\{n: \mathbf{\Theta}_n=\chi\right\}$.
There exists a constant $C_W$ such that for any $W_{\max}$ satisfying $W_{\max}\geq C_W e^{\chi}$, and setting
\begin{equation}
L=L(W_{\max})=\frac{1}{\chi}\left(\log\left(\frac{W_{\max}}{C_W}\right)-(n(\mathbf{\Theta},\chi)-1)\log\left(\frac{1}{\chi}\log\left(\frac{W_{\max}}{C_W}\right)\right)\right), 
\end{equation} 
the combination technique solution satisfies
\begin{align}
\text{Work}[\mathcal{M}_{\I(L(W_{\max}))}]\leq W_{\max},\\
\limsup_{W_{\max}\rightarrow \infty} \frac{Error[\mathcal{M}_{\I(L(W_{\max}))}]}{W_{\max}^{-\mu}(\log(W_{\max})^{(\mu+1)(n(\mathbf{\Theta},\chi)-1})}=C<\infty,
\end{align} 
\end{theorem}
Notice that Theorem \ref{thm:stoc_det} provides a characterization of the asymptotic complexity. As we discuss in Section \ref{sec:numerical}, the pre-asymptotic behaviour can be slightly different. 
Second, the asymptotic complexity depends exclusively on the rates associated to the \textit{spatial} work and error contribution. The asymptotic complexity is not affected by the rates of the stochastic variables. This can be understood noticing that, while both the spatial and stochastic errors decay exponentially with respect to $\alpha_i$ and $\beta_i$, the spatial work grows exponentially with respect to $\alpha_i$, whereas the stochastic work grows only linearly with respect to $\beta_i$.
Third, the asymptotic complexity of the CT provided by Theorem \ref{thm:stoc_det} is identical to the MISC complexity result presented in \cite{haji2016multi}, thus the CT applied to OCPs have the same asymptotic convergence of the MISC method for the solution of PDEs in the context of forward UQ.

\begin{remark}[Infinite dimensional setting]
The complexity analysis presented here cannot be readily generalized to the infinite dimensional setting ($N=\infty$), since both theorems involve constants that explode as $N\rightarrow \infty$. To handle the infinite dimensional case, a refined analysis is required based on the summability of the sequence $\left\{\widetilde{g}_n\right\}_{n=1}^\infty$ in appropriate $l_p$ spaces, following, e.g., \cite{haji2016multi2}. Nevertheless, from the implementation point of view, the CT can be extended to the infinite dimensional setting by modifying Alg \ref{alg:2} in the loop 10-17, so that only a subset of the reduced margin involving a finite number of dimensions is explored at each iteration. This improved algorithm would automatically balance the quadrature error, the spatial discretization error, and the truncation error that is commonly committed by considering only a finite number of random parameters (as in truncated Karhunen–Loève expansions of random fields). 
For more details, we refer to \cite{doi:10.1137/19M1294952} for an analysis concerning truncation error and to \cite{ernst2018convergence,chkifa2014high,seidler2022dimension} for algorithmic aspects.
\end{remark}

\section{Numerical section}\label{sec:numerical}
In this section, we test the effectiveness of the CT on a model problem. 
In all experiments, we consider the OCP \eqref{eq:OCP_example} with force term $\phi=1$ and with the diffusion coefficient 
\begin{equation}\label{eq:diffusion_coefficient}\df(x,\zetab)=e^{\sum_{i=1}^N \zeta_n\lambda_n \psi_n(\mathbf{x})},
\end{equation}
where $\zeta_n\sim \mathcal{U}(-1,1)$ and $\lambda_n=\sqrt{3}e^{-0.6n}$.
The spatial domain is $\mathcal{D}=(0,1)^d$, with $d=1,2$. In the one-dimensional case $\psi_n(x)=\phi_{i_1(n)}(x)$, while in the two-dimensional case $\psi_n(\mathbf{x})=\phi_{i_1(n)}(x_1)\phi_{i_2(n)}(x_2)$, where 
\begin{equation}
\phi_n(x)=\begin{cases} \sin(\frac{n}{2}\pi x) & \text{if } n \text{ is even},\\
\cos(\frac{n-1}{2}\pi x) & \text{if } n \text{ is odd}.
\end{cases}
\end{equation}
The maps $i_j:\mathbb{N}^+\rightarrow \mathbb{N}^+$, $j=1,2$, are detailed in Table 
\ref{tab:i}.
\begin{table}\label{tab:i}
\centering
\begin{tabular}{c c c c c c c c c c c c }
\hline
n & 1 & 2 & 3 & 4 & 5 & 6 & 7 & 8 & 9 & 10 & $\ldots$\\
\hline
$i_1(n)$ & 1 & 2 & 1 & 3 & 2 & 1 & 4 & 3 & 2 & 1 & $\ldots$\\
$i_2(n)$ & 1 & 1 & 2 & 1 & 2 & 3 & 1 & 2 & 3 & 4 & $\ldots$\\
\hline
\end{tabular}
\caption{Definition of the maps $i_1$ and $i_2$.}
\end{table}
Notice that \eqref{eq:diffusion_coefficient} leads to a well-posed state equation, as \[k_{\min}:=e^{-\sum_{n=1}^N \lambda_n}\leq \df(x,\zetab)\leq e^{\sum_{n=1}^N \lambda_n}=:k_{\max},\]
so that Assumptions \ref{ass:well-posedness} and \ref{ass:well-posedness_OCP} are satisfied, even in the limit $N\rightarrow \infty$.
As a quantity of interest, we consider $F(y(\zetab))=\frac{1}{2}\|y(\zetab)-y_d\|^2_{L^2(D)}$ with $y_d:=\prod_{i=1}^d \sin(\pi x_i)$. Note that the precise expression of $y_d$ does not influence the stochastic regularity of the problem since it is deterministic. However, a less regular target state may influence the spatial regularity.

\subsection{Discussion on the validity of the assumptions}\label{sec:validation_of_hypothesis}
In this subsection, we discuss the rational behind the assumptions made in Section \ref{sec:apriori}, and verify them on the model problem described in Example \ref{example:1}.

To satisfy Assumption \ref{ass:mesh_dependence_linear_map}, we consider a spatial mesh obtained as the tensor product of one-dimensional meshes, each characterized by a uniform mesh size $h_{i,\alpha_i}=2^{-\alpha_i-1}$, hence $h_0=2^{-1}$. In one dimension, we employ the Lagrange $\mathbb{P}^1$ finite element space, while in two dimensions we use the $\mathcal{Q}^1$ bilinear finite element space, see \cite{ern2004theory}.
The total number of degrees of freedom on the mesh associated to $\alphab$ is $N_{\alphab}=\prod_{n=1}^D (2^{\alpha_n+1}-1)\leq C_{D} \prod_{n=1}^D 2^{\alpha_n}$.
Further, we consider a linear level-to-nodes map to define the tensor product quadrature rule. The total number of Gauss-Legendre quadrature points is $M_{\betab}=\prod_{n=1}^N \beta_n$.
 
Next, we analyse Assumption \ref{ass:work} on the work contributions.
To compute the optimal control $u^{\alphab,\betab}$, we solve the full-space discrete optimality system (see \cite{vanzan,Matthieu} for a derivation),
\begin{equation}\label{eq:full_space}
\begin{pmatrix}
\mathcal{M} & 0 &\mathcal{A}\\
0 & \nu M_s & -M_s E^\top\\
\mathcal{A} & -E M_s & 0
\end{pmatrix}\begin{pmatrix}
\mathbf{y}\\\mathbf{u}\\ \mathbf{p}
\end{pmatrix}=\begin{pmatrix}
\mathbf{f}_p\\ 0\\\mathbf{f}_y
\end{pmatrix},
\end{equation}
where $\mathcal{A}=\text{diag}\left(w_1A(\bar{\zetab}_1),\dots,w_{M_{\betab}} A(\bar{\zetab}_{M_{\betab}})\right)$ and $A(\bar{\zetab}_n)\in\mathbb{R}^{N_{\alphab}\times N_{\alphab}}$ is the stiffness matrix associated to the quadrature node $\bar{\zetab}_n$ and $w_n$ is its quadrature weight, $M_s\in \mathbb{R}^{N_{\alphab}\times N_{\alphab}}$ is the mass matrix, $\mathcal{M}=\text{diag}\left(w_1M_s,\dots,w_{M_{\betab}} M_s\right)$, and $E=\left(w_1 I_s, \dots, w_{M_{\betab}} I_s\right)^\top$ where $I_s\in \mathbb{R}^{N_{\alphab}\times N_{\alphab}}$ is the identity matrix. Note that the dimensions of $\mathcal{A}$, $\mathcal{M}$ and $E$ depend on the multi-index $(\alphab,\betab)$. We solve this large (symmetric) saddle-point system using MINRES preconditioned by the block diagonal preconditioner analysed in \cite{vanzan}. For a fixed value of the regularization parameter $\nu$, this preconditioner leads to a robust convergence with respect to the mesh size and the number of collocation points, so that the number of Krylov iterations can be considered constant. The major cost of each iteration is the preconditioning of the $2M_{\betab}$ stiffness matrices. Alternatively to \eqref{eq:full_space}, we may perform a Schur complement on $u$ and solve the reduced optimality system
\begin{equation}\label{eq:reduced_space} 
\left(\nu M_s +\sum_{n=1}^{M_{\betab}} w_n(M_s A^{-1}(\bar{\zetab}_n)M_sA^{-1}(\bar{\zetab}_n)M_s)\right)\mathbf{u}=\mathbf{g}.
\end{equation}
Each Conjugate Gradient iteration then requires to invert $2M_{\betab}$ stiffness matrices.
In spite of the full-space/reduced approach used, the overall cost for solving the OCP depends then linearly on the number of collocation points $M_{\betab}$, and possibly nonlinearly (depending on the preconditioner used) on the size of the finite element space, that is,
\begin{equation}\label{eq:bound_on_work_single_OCP}
\text{Work}[u^{\alphab,\betab}]\leq \widehat{C} \left( \prod_{n=1}^N \beta_n\right)  \left(\prod_{n=1}^D 2^{\alpha_n}\right)^{\theta},  
\end{equation}
for a real parameter $\theta$ and constant $\widehat{C}$.
The deterministic and stochastic work contributions thus satisfy
\begin{equation*}
\medmuskip=-0.5mu
\thinmuskip=-0.5mu
\thickmuskip=-0.5mu
\nulldelimiterspace=1pt
\scriptspace=1pt 
\arraycolsep1em 
\begin{aligned}
\Delta W_{\alphab,\bar{\betab}}^{\det}=\sum_{\jb\in \left\{0,1\right\}^D} \text{Work}[u^{\alphab-\jb,\bar{\betab}}] &\leq \widehat{C} \sum_{\jb\in \left\{0,1\right\}^D} \left( \prod_{n=1}^N \bar{\beta}_n\right) \left(\prod_{n=1}^D 2^{\alpha_n-j_n}\right)^{\theta} \leq C^{\det,\bar{\betab}}_{\work} \left(\prod_{n=1}^D 2^{\alpha_n}\right)^{\theta},
\end{aligned}
\end{equation*}
with $C^{\det,\bar{\betab}}_{\work}=\widehat{C}(1+2^{-\theta})^D\left( \prod_{n=1}^N \bar{\beta}_n\right)$ for a fixed $\bar{\betab}$, and
\begin{equation*}
\medmuskip=-1mu
\thinmuskip=-1mu
\thickmuskip=-1mu
\nulldelimiterspace=0.9pt
\scriptspace=0.9pt 
\arraycolsep0.9em 
\begin{aligned}
\Delta W_{\bar{\alphab},\betab}^{\stoc}=\sum_{\jb\in \left\{0,1\right\}^N} \text{Work}[u^{\bar{\alphab},\betab-\jb}] &\leq \widehat{C} \sum_{\jb\in \left\{0,1\right\}^N} \left( \prod_{n=1}^N (\beta_n-j_n)\right)\left(\prod_{n=1}^D 2^{\bar{\alpha}_n}\right)^{\theta}\leq C^{\stoc,\bar{\alphab}}_{\work} \left(\prod_{n=1}^N (\beta_n+1)\right),
\end{aligned}
\end{equation*}
where $C^{\stoc,\bar{\alphab}}_{\work}=\widehat{C}2^N\left(\prod_{n=1}^d 2^{\bar{\alpha}_n}\right)^{\theta}$ for a given $\bar{\alphab}$.
We are left to check \eqref{ass:produc_structure_W}, but indeed 
\begin{equation*}
\medmuskip=-1mu
\thinmuskip=-1mu
\thickmuskip=-1mu
\nulldelimiterspace=0.9pt
\scriptspace=0.9pt 
\arraycolsep0.9em 
\begin{aligned}
\Delta W_{\alphab,\betab}=\sum_{\ib\in \left\{0,1\right\}^{D+N}} \text{Work}\big[ u^{(\alphab,\betab)-\ib}\big]& \leq 
\sum_{\kb\in \left\{0,1\right\}^D}\sum_{\jb\in  \left\{0,1\right\}^N} \widehat{C} \left(\prod_{n=1}^D 2^{\alpha_n-j_n}\right)^{\theta}\left( \prod_{n=1}^N (\beta_n-j_n)\right)\\
&\leq  \widehat{C} 2^{N} \left( \prod_{n=1}^N (\beta_n+1)\right) \sum_{\kb\in \left\{0,1\right\}^D} \left(\prod_{n=1}^D 2^{\alpha_n-j_n}\right)^{\theta}\\
&\leq  C_{\work} \left( \prod_{n=1}^N (\beta_n+1)\right) \left(\prod_{n=1}^D 2^{\alpha_n}\right)^{\theta},
\end{aligned}
\end{equation*}
with $C_{\work}=2^{N}(1+2^\theta)^D\widehat{C}$. 

Notice that \eqref{eq:ass:1} and \eqref{ass:produc_structure_W} are trivially satisfied setting $\widetilde{\gamma}_i=\theta$, for every $i$.
In our numerical experiments, we used the full-space approach and invert directly the $2M_{\betab}$ matrices using the built-in Matlab sparse banded matrices solver, and found that $\theta=1$ gives a good description of the increase of the computational time with respect to the size of the finite element space.

Next, we focus on Assumption \ref{ass:error}.
On the one hand, \eqref{eq:ass:2} is a direct consequence of classical finite element error analysis combined with the spatial combination technique convergence theory, provided that $u$ admits sufficient Sobolev mixed-order regularity \cite{griebel1990combination,combination_te}. However, for general domains and/or with control constraints, the control might not enjoy such high regularity. In these cases, the CT could still be applied either on a fixed spatial mesh, or by letting $\alpha$ be a scalar setting the overall mesh size, as in standard multilevel approaches \cite{harbrecht2012multilevel,van2019robust}. 

Concerning \eqref{eq:ass:4}, it is well-known that hierarchical surpluses satisfy such hypothesis in the context of sparse grids interpolants for forward UQ problems, provided that the solution map of the random PDE is holomorphic in a Bernstein ellipse in the complex plane \cite{nobile2016convergence,haji2016multi,cohen2015approximation}. In our setting, we intuitively remark that \eqref{eq:ass:4} requires the map $\zetab\in \Gamma \rightarrow p(\zetab)\in V$ to be holomorphic in a region of the complex plane, which in turn requires holomorphic regularity of the quantity of interest $F$ and of the inverse of the operator $e$ with respect to $\zetab$.

Finally, \eqref{ass:produc_structure_E} assumes a product structure of the decay of the mixed spatial and stochastic surpluses. For forward UQ problems, this property has been extensively studied and exploited in several works, see, e.g., \cite{harbrecht2012multilevel,teckentrup2015multilevel,MIMC,haji2016multi,haji2016multi2,beck2019iga}. The rigourous extension of these results to the present context requires a highly technical theoretical analysis that is beyond the scope of this work and it is the subject of a forthcoming manuscript.

Here, we limit to verify \eqref{eq:ass:2}, \eqref{eq:ass:4} and \eqref{ass:produc_structure_E} numerically, and we estimate the rates $\widetilde{r}_i$ and $\widetilde{g}_i$. To do so, we fix a multi-index $\betab=\bar{\betab}$, set $\alphab=\unob+j\bar{\alphab}$, and fit the decay of $\Delta E_{\alphab,\bar{\betab}}$ for some different values of $j$. Similarly, we fix a $\alphab=\bar{\alphab}$ and set $\betab=\unob+j\bar{\betab}$ and fit the decay of $\Delta E_{\bar{\alphab},\betab}$.
The mixed decays are fitted with the same procedure.
The results are reported in Figures \ref{Fig:fitted}, which clearly show that the assumptions are verified for the model problem. In our setting, we found $\widetilde{r}_i=2$, for all $i$, while the rates $\widetilde{g}_i$ are reported in Table \ref{tab:rates}. Notice that the estimated rates $\widetilde{g}_i$ may depend on the geometrical mesh used, but the dependence is quite mild (see \cite{beck2019iga} for a procedure to update the estimate of $\widetilde{g}_i$ on the fly).
\begin{figure}
\centering
\includegraphics[scale=0.34]{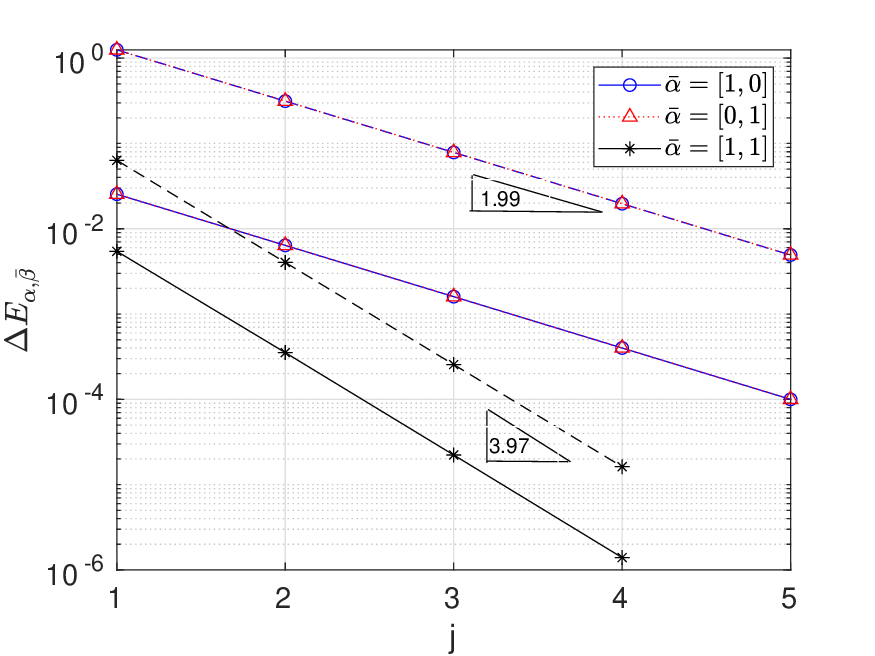}
\includegraphics[scale=0.34]{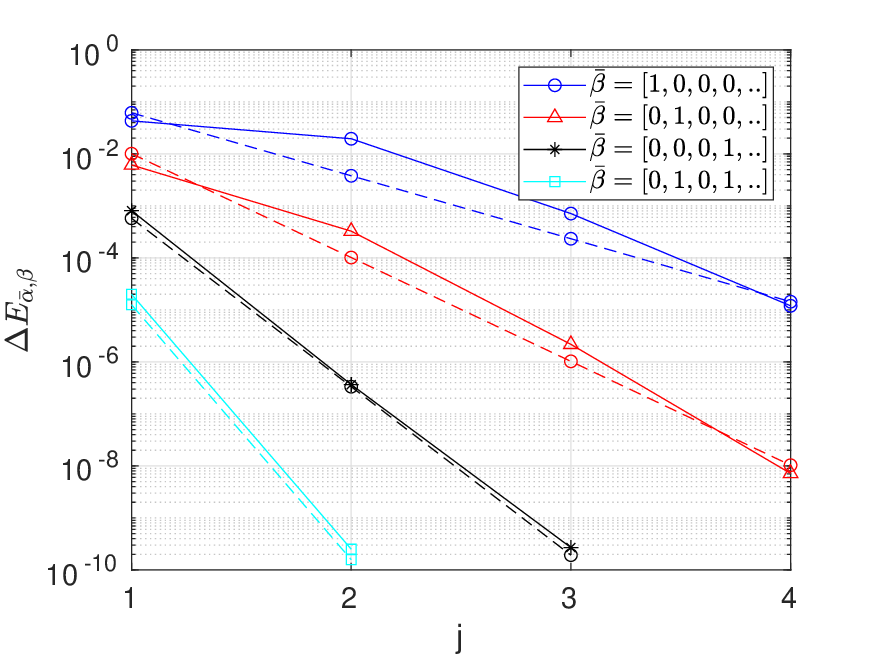}
\includegraphics[scale=0.34]{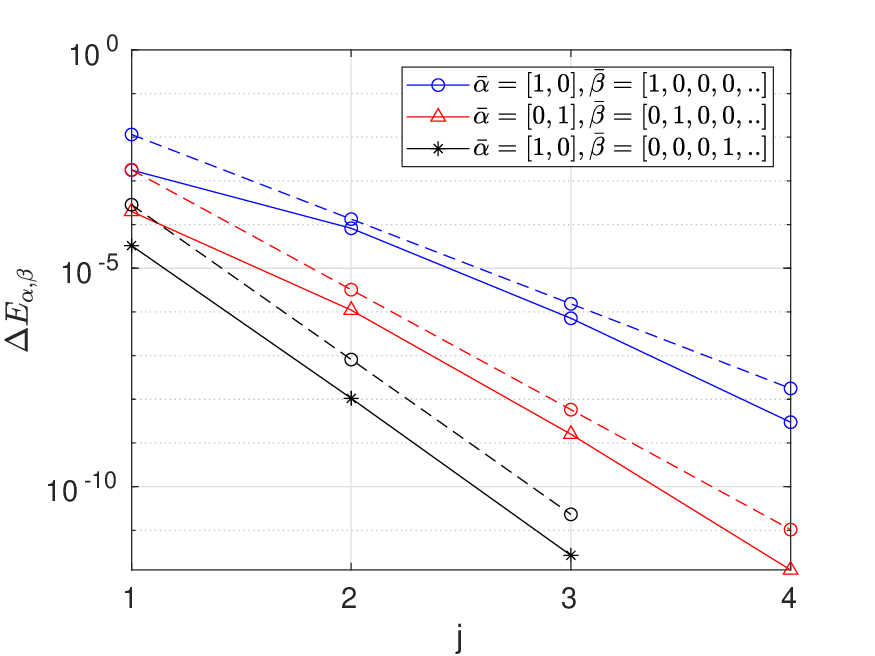}
\caption{Numerical validation of Assumption \eqref{eq:ass:2}, \eqref{eq:ass:4} and \eqref{ass:produc_structure_E} on the decay of the error contributions. Spatial error contributions with $\bar{\betab}=\bm{1}$ (left), stochastic error contributions with $\bar{\alphab}=(3,3)$ (center) for the first four random variables, and mixed spatial and stochastic error contributions (right). The solid lines are based on computed values, the dashed lines are the fitted ansatzes.}\label{Fig:fitted}
\end{figure}
\begin{table}
\centering
\begin{tabular}{|c | c | c | c | c | c | c | c | c | c |c}
\hline
$\widetilde{g}_1$ & $\widetilde{g}_2$ & $\widetilde{g}_3$ & $\widetilde{g}_4$ & $\widetilde{g}_5$ & $\widetilde{g}_6$ & $\widetilde{g}_7$ & $\widetilde{g}_8$ & $\widetilde{g}_9$ & $\widetilde{g}_{10}$ & $\ldots$\\\hline
2.78 & 4.59 & 5.79 & 7.45 & 8.18 & 9.85 & 10.97 & 12.98 & 14.18 & 14.57 & $\ldots$\\
\hline
\end{tabular}
\caption{Fitted rates $\widetilde{g}_i$ for the first 10 random variables in the expansion of \eqref{eq:diffusion_coefficient}.}\label{tab:rates}
\end{table}

\subsection{Numerical tests}
We now show the performance of the CT to solve \eqref{eq:OCP_example}. To construct the multi-index sets, we rely on the sparse-grid Matlab Kit \cite{back2011stochastic}.
We first consider the CT approximation \eqref{eq:CT_onlystoch} applied exclusively on the quadrature formula of the objective functional. For every multi-index $\betab$, we use the same physical discretization determined by a fixed $\bar{\alphab}$. 
More specifically, we compare five methods:
\begin{itemize}
\item[Meth. 1] An ``a-priori incremental" algorithm in which $\I$ is built adaptively by using the ansatzes \eqref{eq:ass:3} and \eqref{eq:ass:4} to compute the profits for all multi-indices in the reduced margin. The parameters $\left\{\widetilde{g}_n\right\}_n$ are estimated beforehand, and this cost is not considered. This algorithm corresponds to Alg. \ref{alg:2}.
\item[Meth. 2] The well-known adaptive algorithm \cite{gerstner2003dimension} in which $\I$
is built adaptively computing numerically $\Delta E_{\bar{\alphab},\betab}$ (and thus also the profits $P_{\bar{\alphab},\betab}$) for all multi-indices in the reduced margin. We keep track only of the work done to solve the OCPs that actually contribute to the CT solution \eqref{eq:CT_estimator}, thus neglecting the cost of exploring the reduced margin (which might dominate the overall cost). In this way, the adaptive algorithm can be considered as a benchmark for the a-priori CT. This algorithm corresponds to Alg. \ref{alg:1}.
\item[Meth. 3] An a-priori algorithm based on the CT formulation \eqref{eq:CT_estimator_equiv}. In particular, given the multi-index set $\I^k$ built at the $k$-th step of the a-priori incremental algorithm (Meth. 1), we check for which $\betab$, $c_{\betab}\neq 0$, and consequently compute the approximation using \eqref{eq:CT_estimator_equiv}. The approximation will be identical to the solution of Meth. 1. However, the work performed will be less, as several multi-indices $\betab$ do not contribute to the approximation.
\item[Meth. 4] An anisotropic tensor product approximation: given a sequence of levels $L$, we set $\beta_n=1+\lfloor \frac{L}{\widetilde{g}_n}\rfloor$. 
Notice that the anisotropic tensor product approximation can be recast in the CT framework, using the multi-index set $\I(L)=\left\{\betab \in \mathbb{N}^N: \max_n \widetilde{g}_n(\beta_n-1)\leq L\right\}$, see \cite{back2011stochastic}. 
\item[Meth. 5] A standard sparse grid method (see Remark \ref{remark:sparse_grids}) based on the multi-index set $\I(L)=\left\{\betab\in \mathbb{N}^N: \sum_{i=1}^N \widetilde{g}_n(\beta_n-1)\leq L\right\}$ with Gauss-Legendre nodes and a linear level-to-node map.
\end{itemize}

Fig. \ref{Fig:onlystoch} shows the complexity (error vs work) of the different methods for $N=2,6,10$ random variables and $d=2$. 
The work reported in the x-axis corresponds to $W=\sum_{\betab\in\I} M_{\betab}$ for Meth. 1 and Meth. 2, 
$W=\sum_{\betab\in\I: c_{\betab}\neq 0}M_{\betab}$ for Meth. 3,
$W=M_{\betab}$ for Meth. 4, and $W=|\widetilde{\Lambda}_{\betab}|$, i.e. the number of nodes involved in the sparse grid quadrature, for Meth. 5.
The error corresponds to the $L^2$ norm between the current approximation and a reference overkilled solution computed with Alg. 2.

Notice that the a-priori algorithms converge very similarly to the adaptive algorithm, thus confirming the effectiveness of the a-priori construction based on the model-fitted ansatzes.
Generally, the CT needs a few random variables in order to be more effective than anisotropic full tensor approximations, as sparse approximations are more efficient for sufficiently high-dimensional problems \cite{bungartz2004sparse}.

The convergence of the CT is essentially unchanged moving from $N=6$ to $N=10$ random variables. Indeed, the random variables are increasingly less important, being weighted in \eqref{eq:diffusion_coefficient} by $\lambda_n$ which tends to zero exponentially, and their rates $\widetilde{g}_i$ are increasingly larger (see Table \ref{tab:rates}). Hence, the addition of a new random variable requires only to compute quite cheap hierarchical surplus $\Delta^{\betab} u$. This can be observed in the left and center panel of Fig \ref{Fig:set}, which show three components of the multi-indices $\betab$ included by Meth. 2 for $d=2$ and $N=10$. We clearly observe that whenever $\beta_9$ or $\beta_{10}$ are greater than 1, all the remaining components of $\betab$ are very small, which implies that the corresponding OCP will involve very few quadrature nodes. 
In contrast, the addition of a new random variable, let's say the $N+1$-th, implies that the number of collocation points $M_{\beta}$ for the anisotropic tensor product approximation is multiplied by $\beta_{N+1}=1+\lfloor\frac{L}{\widetilde{g}_{N+1}}\rfloor$, and thus the single OCP that has to be solved becomes steadily larger. 
The classical sparse grid approach shows a computational complexity similar to that of the CT in all cases. However, due to the presence of negative weights, the full-space optimality system had to be solved with GMRES. The additional cost and memory usage due to the larger Krylov subspace have not been taken into account in the figure.
Finally, the right panel of Fig. \ref{Fig:onlystoch} verifies instead the simplified complexity estimate \eqref{eq:decay_simplified}, as the error decreases linearly in the log-log plot.

\begin{figure}
\centering
\includegraphics[scale=0.32]{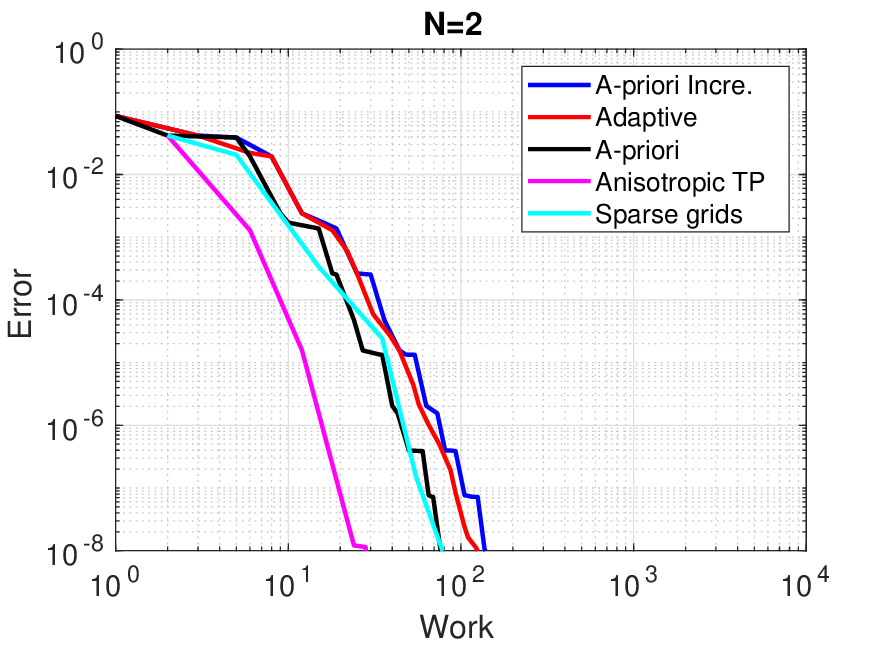}
\includegraphics[scale=0.32]{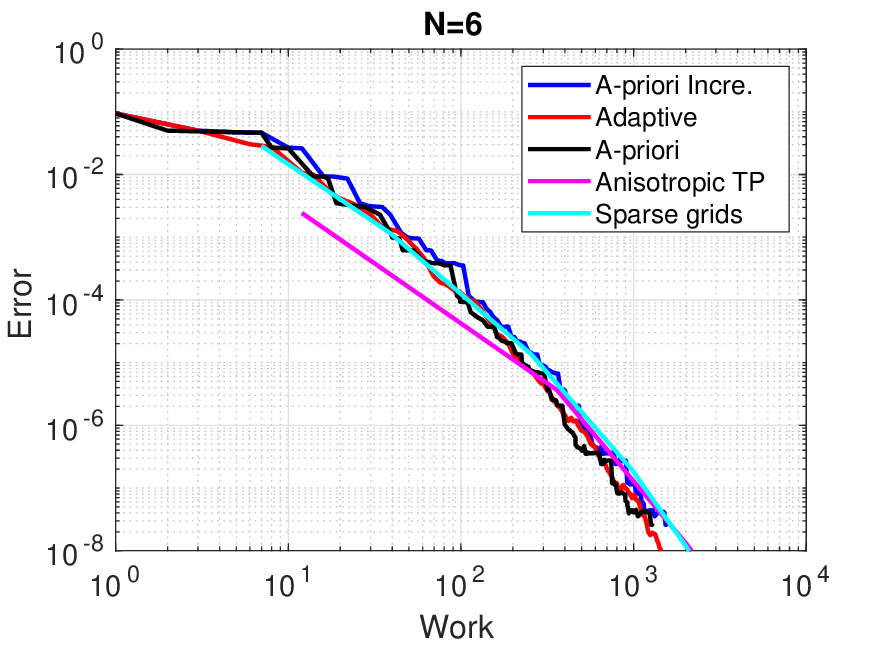}
\includegraphics[scale=0.32]{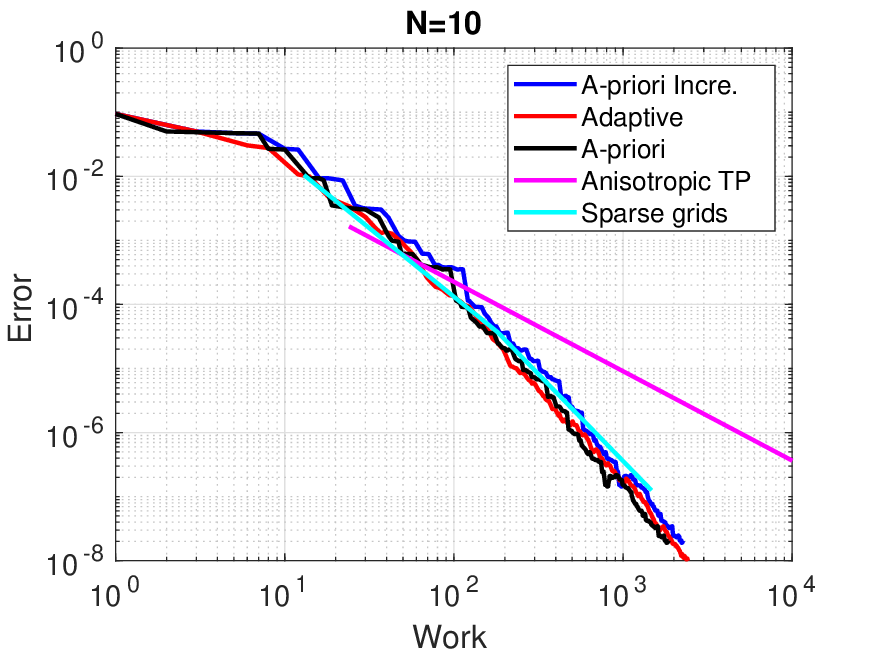}
\caption{Convergence behaviour of the different methods for $N=2,6,10$ random variables.}\label{Fig:onlystoch}
\end{figure}

\begin{figure}
\centering
\includegraphics[scale=0.32]{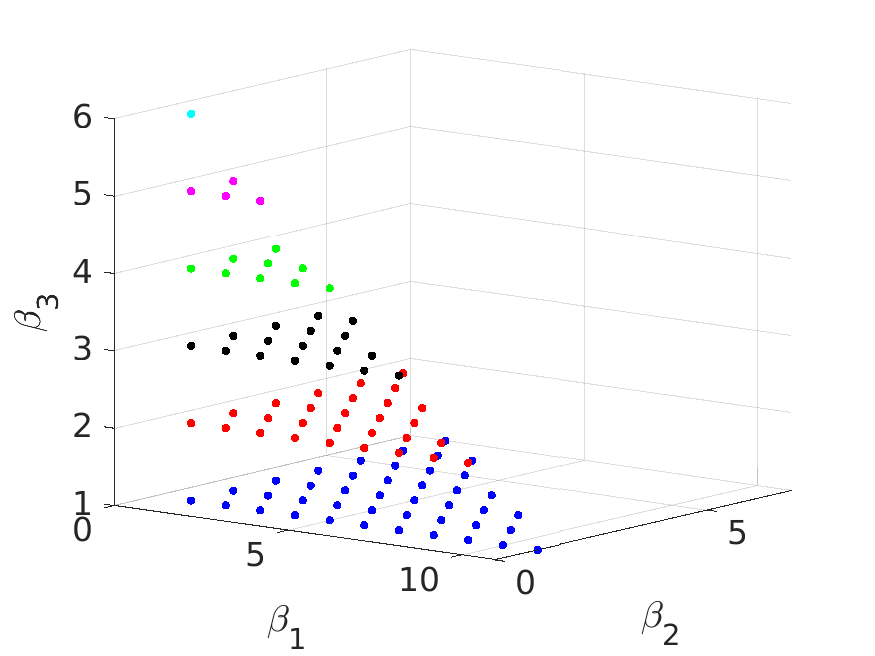}
\includegraphics[scale=0.32]{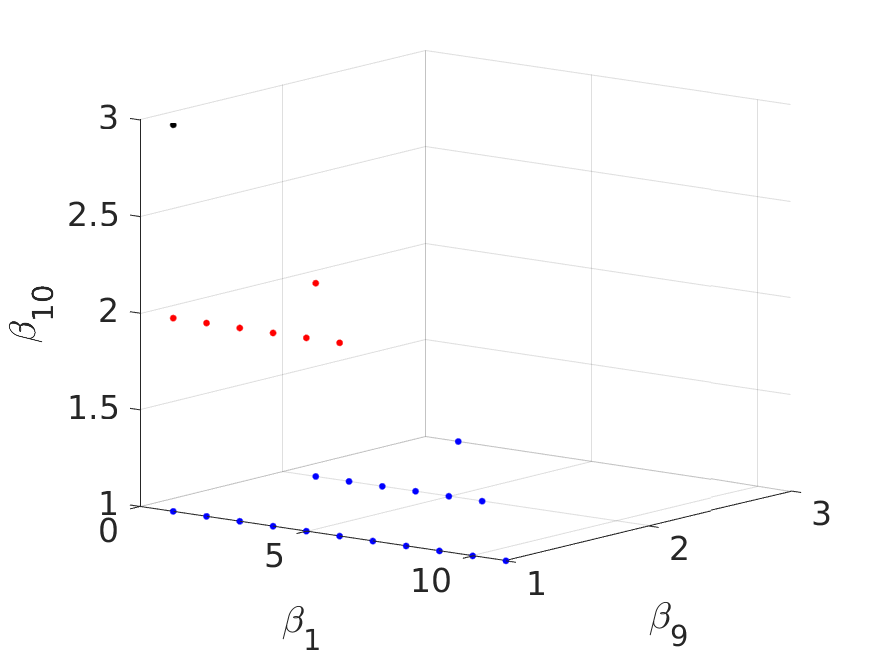}
\includegraphics[scale=0.32]{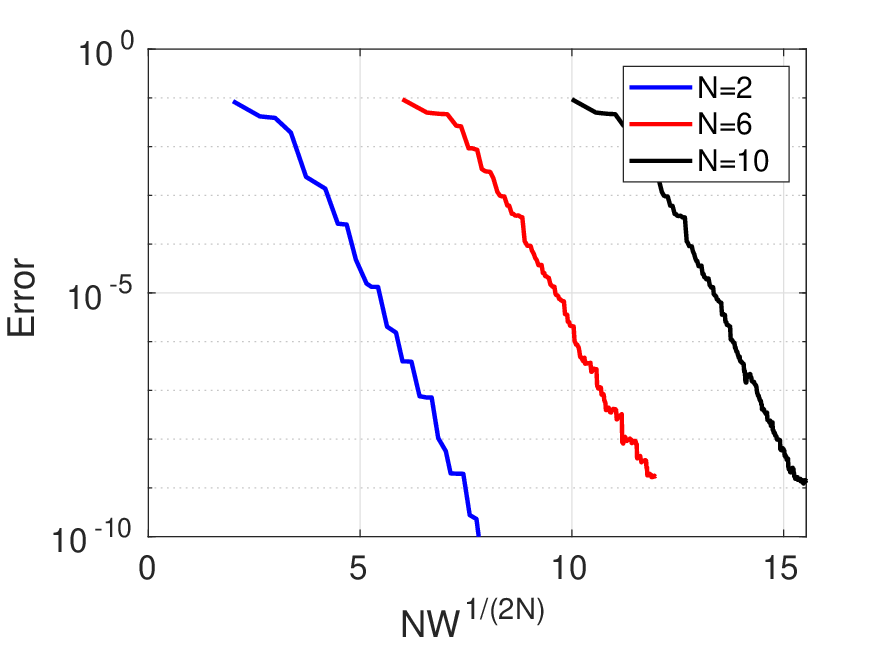}
\caption{The left and center panel show the sparsity pattern of $(\beta_1,\beta_2,\beta_3)$ and $(\beta_1,\beta_9,\beta_{10})$ of Alg. 2 for $d=2$ and $N=10$. The right panel verifies numerically the simplified convergence estimate \eqref{eq:decay_simplified}.}\label{Fig:set}
\end{figure}

Next, we study the complexity of the CT applied to both spatial and stochastic variables.
Fig. \ref{Fig:CTfull} shows the convergence for the model problem set in a one-dimensional domain (top row) and a two dimensional domain (bottom row). For Meth. 4 and 5, we consider a sequence of spatial meshes obtained by halving simultaneously the mesh size in each spatial direction, and an increasing sequence of levels $L$ for the stochastic quadrature.
The work reported in the x-axis corresponds to $W=\sum_{(\alphab,\betab)\in\I} \left(\prod_{n=1}^D 2^{\alpha_n+1}\right) M_{\betab}$ for Meth. 1 and Meth. 2, 
$W=\sum_{\betab\in\I: c_{\betab}\neq 0}\left(\prod_{n=1}^D 2^{\alpha_n+1}\right) M_{\betab}$ for Meth. 3, $W=\left(\prod_{n=1}^D 2^{\alpha_n+1}\right) M_{\betab}$ for Meth. 4, and $W=\left(\prod_{n=1}^D 2^{\alpha_n+1}\right)|\widetilde{\Lambda}_{\betab}|$ for Meth. 5.
We calculate a reference solution by running Meth. 2 with a very small tolerance, and we linearly interpolate all the solutions computed (which live on different meshes) on a very fine mesh obtained by taking, for each physical dimension, the smallest mesh size used by Meth. 2. In two dimensions, the reference solution has more than $2\cdot 10^8$ degrees of freedom.
To compute the error, we interpolate the current approximations on the reference mesh and calculate there the $L^2$ norm of the difference with the reference solution using a mass lumped matrix. Notice that the interpolation step does not introduce errors as the finite element basis functions are linear in 1D and bilinear in 2D, and the meshes are nested.

As $\widetilde{\gamma}_j=1$ and $\widetilde{r}_j=2$ for $j=1,2$, Theorem \ref{thm:stoc_det} predicts an \textit{asymptotic} complexity of $\text{Work}^{-2}$ for  the 1D setting and $\text{Work}^{-2}\log(\text{Work})^3$ for the 2D setting.
For few random variables (e.g. $N=2$) the asymptotic complexity is rapidly attained in both settings while for larger number of random variables pre-asymptotic effects are non negligible, and the asymptotic regime is achieved for larger values of $W$. Meth. 5 shows a worse complexity behaviour than that of the CT, especially for $d=2$ where we clearly see the benefits of the mixed spatial regularity.
Finally, we emphasize that both the anisotropic tensor product and sparse grids approximations are not able to achieve smaller errors on our workstation since the size of their optimality system quickly saturates the memory storage. By solving several OCPs characterized by smaller optimality systems, the CT is able to overcome the limitations set by the workstation and achieve much smaller tolerances.

\begin{figure}
\centering
\includegraphics[scale=0.324]{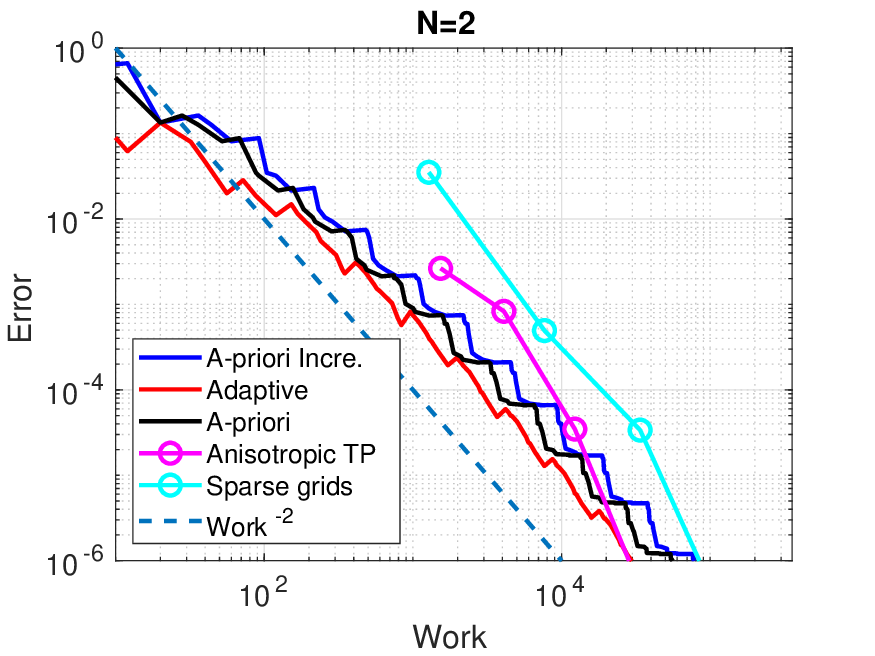}
\includegraphics[scale=0.324]{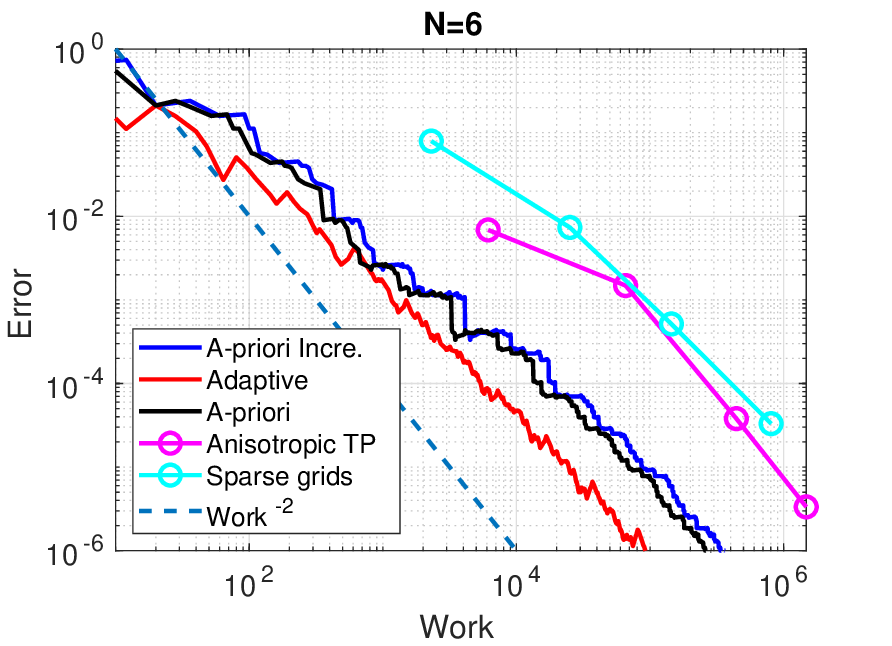}
\includegraphics[scale=0.324]{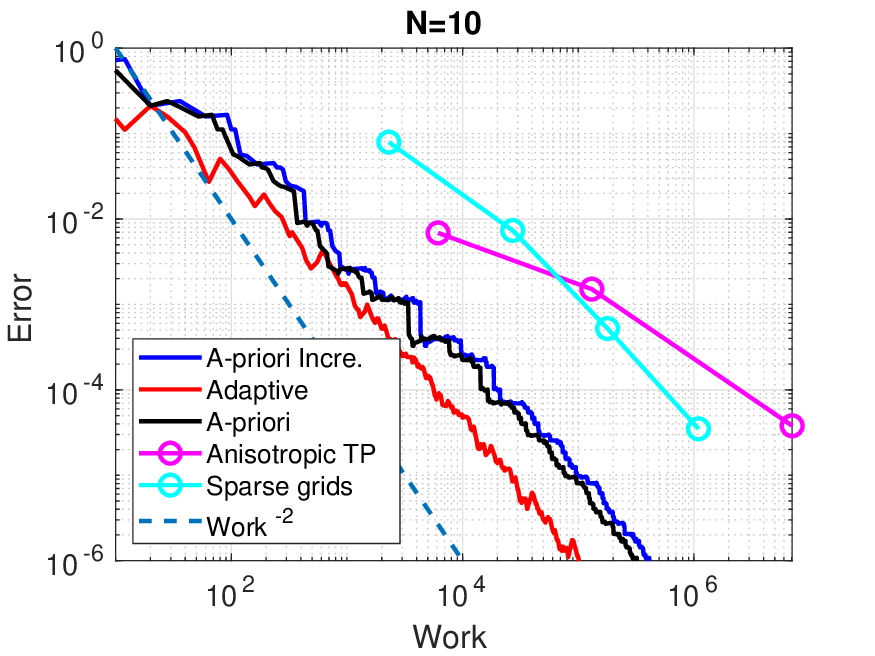}
\includegraphics[scale=0.324]{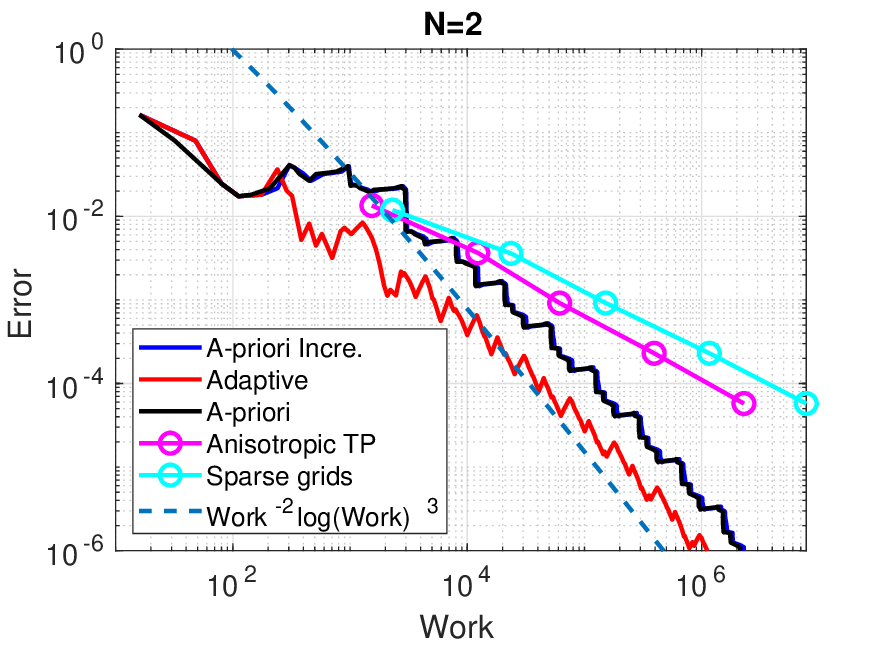}
\includegraphics[scale=0.324]{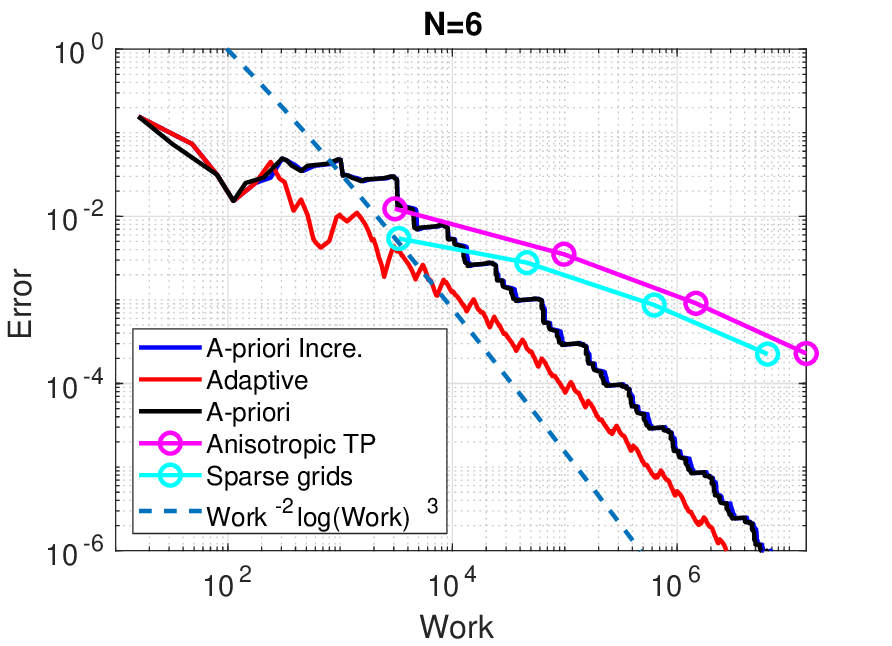}
\includegraphics[scale=0.324]{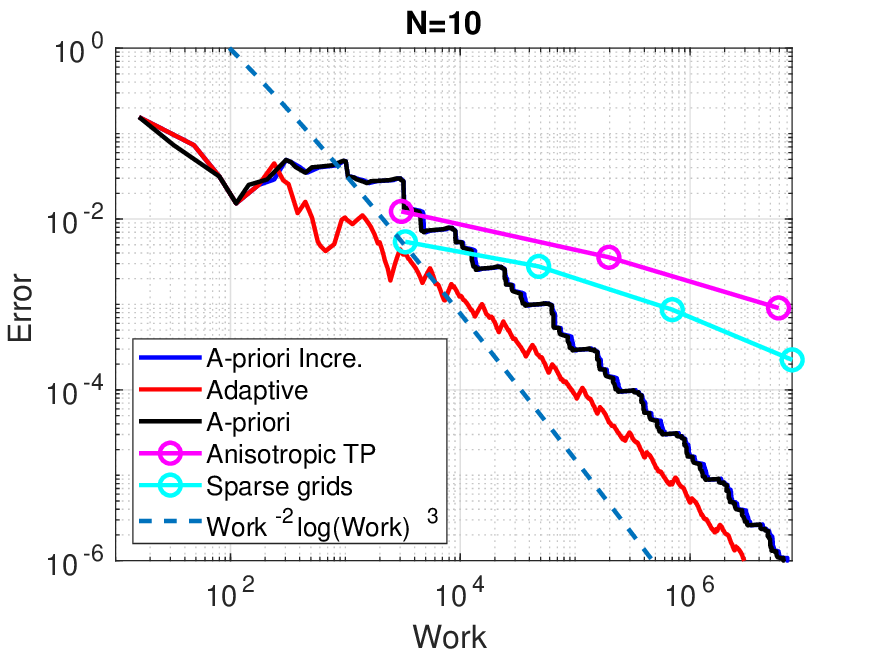}
\caption{Convergence of the combined $(\alphab,\betab)$ CT for a one-dimensional physical problem (top row) and a two-dimensional one (bottom row)}\label{Fig:CTfull}
\end{figure}

\section{Conclusions}
In this work, we proposed two variants of the combination technique to solve efficiently optimal control problems under uncertainty. The first one is based on the combination of solutions of several OCPs discretized using different tensor product quadrature formulae, but on the same spatial mesh. The approach allows to reduce the computational effort for high dimensional problems, while avoiding the inconveniences of a direct sparse grid approximation of the objective functional.
Then, we applied the combination technique both to the spatial and stochastic variables. Such procedure allows to automatically balance the error due to the spatial and stochastic discretization according to a profit rule and it permits to further reduce the computational cost by moving most computations on coarser grids.
The generalization of this work to an infinite sequence of random variables, including a theoretical analysis for the decay of the error contributions, is currently ongoing and will be the subject of a forthcoming manuscript. Further efforts will be devoted to test the combination technique for OCPs involving control constraints and more general risk measures. The problem formulation \eqref{eq:OCP} considered here is already sufficiently general to cover some common risk-measures such as the mean-variance model. Of particular interest in risk-adverse contexts is the CVaR (Conditional Value at Risk). Although the minimization of the CVaR can be formulated as a double minimization problem leading to an optimality condition of the form of \eqref{eq:optimality_condition}, it also involves a discontinuous function of the random vector $\zetab$ acting on the right-hand side of the adjoint equation which might reduce the stochastic regularity. Thus, a CT based on a multilevel (Quasi)-Monte Carlo quadrature seems promising as it requires less smoothness than the stochastic collocation method.

\section{Appendix}
In this appendix, we detail the proofs of Theorem \ref{thm:stoc} and \ref{thm:stoc_det}. Both proofs are based on a direct counting argument, and in particular the proof of Theorem \ref{thm:stoc_det} is an adaptation of the proof of Theorem 1 in \cite{haji2016multi} to a different multi-index set.

For a $\mathbf{x}=(x_1,\dots,x_N)\in \mathbb{R}^N$, let $|\mathbf{x}|:=\sum_{n=1}^N |x_i|$ and $\log(\mathbf{x})=(\log(x_1),\dots,\log(x_N))$.
The following technical Lemmae are needed.
\begin{lemma}[Lemma 4 in \cite{haji2016multi}]\label{lemma:1}
Let $f:(1,\infty)^N\rightarrow \mathbb{R}$ and $g:(1,\infty)^N\rightarrow \mathbb{R}_{+}$. If $f$ and $g$ are increasing, then
\[\sum_{\alphab\in \mathbb{N}^N_{+}: f(\alphab)\leq 0} g(\alphab)\leq \int_{\mathbf{x}\in (1,\infty)^N: f(\mathbf{x}-\unob)\leq 0} g(\mathbf{x})d\mathbf{x}.\]
If $f$ and $g$ are decreasing, then
\[\sum_{\alphab\in \mathbb{N}^N_{+}: f(\alphab)\leq 0}g(\alphab)\leq \int_{\mathbf{x}\in (1,\infty)^N: f(\mathbf{x})\leq 0} g(\mathbf{x}-\unob)d\mathbf{x}.\]
\end{lemma}
\begin{lemma}\label{lemma:2}
The following bound holds true
\[\int_{\mathbf{x}\in (0,\infty)^N: |\mathbf{x}|\leq H} \mathbf{x}\;d\mathbf{x}= \frac{H^{2N}}{(2N)!},\]
\end{lemma}
\begin{proof}
Perform the change of variables $\mathbf{t}=\frac{\mathbf{x}}{H}$ and use an induction argument over $N$.
\end{proof}

\begin{lemma}\label{lemma:3}
The following bound holds true
\[\int_{\mathbf{x}\in (0,\infty)^N: |\mathbf{x}|\geq L} e^{-|\mathbf{x}|+|\log(\mathbf{x})|} d\mathbf{x}\leq e^{-L}(L+1)^{2N-1}.\]
\end{lemma}
\begin{proof}
The proof based on induction is part of the proof of Lemma 5 in \cite{haji2016multi}.
\end{proof}

\begin{lemma}[Lemma 7 in \cite{haji2016multi}]\label{lemma:4}
Let $k\in \mathbb{N}$, $\mathbf{a}=(a_1,\dots,a_D)\in \mathbb{R}_+^D$ and $L>|\mathbf{a}|$. Then,
\[\int_{\left\{\mathbf{x}\in \mathbb{R}_+^D: |\mathbf{x}|\leq L\right\}} e^{\mathbf{a}\cdot \mathbf{x}}(L-|\mathbf{x}|)^kd\mathbf{x}\leq \mathcal{U}_D(\mathbf{a},k)e^{\max(\mathbf{a})L}L^{\eta(\mathbf{a},\max(\mathbf{a}))-1},\]
where $\eta(\mathbf{a},\max(\mathbf{a})):= \#\left\{i: a_i=\max(\mathbf{a})\right\}$ and $\mathcal{U}_D(\mathbf{a},k)$ is a constant independent on $L$.
\end{lemma}

\begin{lemma}[Lemma B.3 in \cite{MIMC}]\label{lemma:5}
Let $k\in \mathbb{N}$, $\mathbf{a}=(a_1,\dots,a_N)\in \mathbb{R}_+^D$ and $L>|\mathbf{a}|$. Then,
\[\int_{\left\{\mathbf{x}\in \mathbb{R}_+^D: |\mathbf{x}|>L\right\} } e^{-\mathbf{a}\cdot \mathbf{x}} \leq \mathcal{B}_D(\mathbf{a})e^{-\min(\mathbf{a})L}L^{\eta(\mathbf{a},\min(\mathbf{a}))-1},\]
where $\mathcal{B}_D(\mathbf{a})$ is a constant independent on $L$.
\end{lemma}

\subsection{Proof of Theorem \ref{thm:stoc}}\label{appendix:thm:stoc}
\begin{theorem}
There exist constants  $C_1$ and $C_2$ such that for any $W_{\max}$ satisfying $W_{\max}\geq \frac{|\widetilde{\gb}|^{2N}C_1}{(2N)!}$, and choosing $\widehat{L}=\sqrt[2N]{\frac{W_{\max}(2N)!}{C_1}}-|\widetilde{\gb}|$,
\begin{align*}
\text{Work}\big[\mathcal{M}_{\I(\widehat{L})}(u)\big]&\leq W_{\max},\\
\text{Error}\big[\mathcal{M}_{\I(\widehat{L})}(u)\big]&\leq C_2 e^{-\sqrt[2N]{\frac{W_{\max}(2N)!}{C_1}}}\left(\sqrt[2N]{\frac{W_{\max}(2N)!}{C_1}}+1\right)^{2N-1}.
\end{align*}
\end{theorem}

\begin{proof}
We start with the work estimate. The total work satisfies
\begin{align*}
\text{Work}\big[\mathcal{M}_{\I(L)}(u)\big]= \sum_{\betab\in \I(L)} \Delta W^{\stoc}_{\bar{\alphab},\betab}&\leq C_{\work}^{\stoc,\bar{\alphab}} \sum_{\left\{\betab\in \mathbb{N}_{+}^{N}:\ \widetilde{\gb}\cdot (\betab+\unob) + |\log(\betab+\unob)|\leq L\right\}} e^{|\log(\betab+\unob)| }\\
&\leq C_{\work}^{\stoc,\bar{\alphab}} \int_{\left\{\betab\in \otimes_{i=1}^N (1,\infty):\ \widetilde{\gb}\cdot \betab+|\log(\betab)|\leq L\right\}} e^{|\log(\betab+\unob)|}d\betab\\
\end{align*}
where the last inequality follows from Lemma \ref{lemma:1}.

Next, we perform the change of variable $t_j=\widetilde{g}_j \beta_j +\log(\beta_j)$, set $\beta_j=q_j(t)$, with $q_j(t)\leq \frac{t_j}{\widetilde{g}_j}$,
\begin{align*}
\text{Work}\big[\mathcal{M}_{\I(L)}(u)\big]&\leq C_{\work}^{\stoc,\bar{\alphab}} \int_{\left\{\mathbf{t}\in \otimes_{i=1}^N (\widetilde{g}_j,\infty)^N:\ |\mathbf{t}|\leq L\right\}} \prod_{i=1}^N \frac{q_j(t)+1}{\widetilde{g}_jq_j(t)+1} q_j(t) d\mathbf{t}\\
&\leq C_{\work}^{\stoc,\bar{\alphab}} \int_{\left\{\mathbf{t}\in \otimes_{i=1}^N (\widetilde{g}_j,\infty)^N:\ |\mathbf{t}|\leq L\right\}} \prod_{i=1}^N \frac{\frac{t_j}{\widetilde{g}_j}+1}{t_j+1} \frac{t_j}{\widetilde{g}_j} d\mathbf{t} \\
&\leq C_1 \int_{\left\{\mathbf{t}\in \otimes_{i=1}^N (\widetilde{g}_j,\infty)^N:\ |\mathbf{t}|\leq L\right\}} \prod_{i=1}^N (t_j+\widetilde{g}_j)d\mathbf{t}\\
&\leq C_1\int_{\left\{\widetilde{\mathbf{t}}\in \otimes_{i=1}^N (0,\infty)^N:\ |\widetilde{\mathbf{t}}|\leq L+|\widetilde{\gb}|\right\}} \widetilde{\mathbf{t}}d\widetilde{\mathbf{t}}\\
&\leq C_1 \frac{(L+|\widetilde{\gb}|)^{2N}}{(2N)!},
\end{align*}
where $C_1=C_{\work}^{\stoc,\bar{\alphab}}\prod_{i=1}^N\left(\frac{1}{\widetilde{g}^2_j}\right)$ and in the last step we used Lemma \ref{lemma:2}.

We next consider the error term,
\begin{align*}
&\text{Error}\big[\mathcal{M}_{\I(L)}(u)\big]=\sum_{\betab\notin I(L)} \Delta E_{\bar{\alphab},\betab}^{\stoc}\\
&\leq C^{\stoc,\bar{\alphab}}_{\error} \sum_{\left\{\betab\in \mathbb{N}^N_{+}:\ \widetilde{\gb}\cdot (\betab+\unob)+\log(\betab+1)>L\right\}} e^{-\widetilde{\gb}\cdot(\betab+\unob)}\\
&\leq C^{\stoc,\bar{\alphab}}_{\error}\int_{\left\{\betab\in \otimes_{i=1}^N (1,\infty)^N:\ \widetilde{\gb}\cdot (\betab+\unob) +\log(\betab+\unob)>L\right\}} e^{-\widetilde{\gb}\cdot \betab}d\betab\\
&= C^{\stoc,\bar{\alphab}}_{\error}e^{|\widetilde{\gb}|} \int_{\left\{\betab\in \otimes_{i=1}^N (1,\infty)^N:\ \widetilde{\gb}\cdot (\betab+\unob)+\log(\betab+\unob)>L\right\}} e^{-\widetilde{\gb}\cdot (\betab+\unob) +|\log(\betab+\unob)|-|\log(\betab+\unob)|}d\betab.
\end{align*}
The change of variable $\widetilde{g}_i(\beta_i+1) +\log(\beta_i+1)=t_j$, $q(t_j)=\beta_i$, with $q(t_j)\leq \frac{t_j}{\widetilde{g}_j}$, leads to

\begin{align*}
\text{Error}\big[\mathcal{M}_{\I(L)}(u)\big]&\leq C^{\stoc,\bar{\alphab}}_{\error}e^{|\widetilde{\gb}|} \int_{\left\{\mathbf{t}\in \otimes_{i=1}^N (2\widetilde{g}_j+\log(2),\infty)^N:\ |\mathbf{t}|>L\right\}} e^{-|\mathbf{t}|}\left(\prod_{i=1}^N \frac{(q(t_i)+1)^2}{\widetilde{g}_i(q(t_i)+1)+1}\right)d\mathbf{t}\\
&\leq C^{\stoc,\bar{\alphab}}_{\error}e^{|\widetilde{\gb}|} \int_{\left\{\mathbf{t}\in \otimes_{i=1}^N (2\widetilde{g}_j+\log(2),\infty)^N:\ |\mathbf{t}|>L\right\}} e^{-|\mathbf{t}|}\left(\prod_{i=1}^N \frac{t_j+\widetilde{g}_j}{\widetilde{g}_j^2}\right)d\mathbf{t}\\
&\leq C^{\stoc,\bar{\alphab}}_{\error}\left(\prod_{i=1}^N \frac{e^{\widetilde{g}_i}}{\widetilde{g}^2_j}\right) \int_{\left\{\mathbf{t}\in \otimes_{i=1}^N (2\widetilde{g}_j+\log(2),\infty)^N:\ |\mathbf{t}|>L\right\}} e^{-|\mathbf{t}|+|\log(\mathbf{t}+\mathbf{\widetilde{g}})|}d\mathbf{t}\\
&\leq C_2 \int_{\left\{\mathbf{x}\in \otimes_{i=1}^N (3\widetilde{g}_j+\log(2),\infty)^N:\ |\mathbf{x}|>L+|\widetilde{\gb}|\right\}} e^{-|\mathbf{x}|+|\log(\mathbf{x})|}d\mathbf{x}\\
&\leq C_2 \int_{\left\{\mathbf{x}\in \otimes_{i=1}^N (0,\infty)^N:\ |\mathbf{t}|>L+|\widetilde{\gb}|\right\}} e^{-|\mathbf{x}|+|\log(\mathbf{x})|}d\mathbf{x}\\
&\leq C_2 e^{-L-|\widetilde{g}|}(L+1+|\widetilde{g}|)^{2N-1},
\end{align*}
where we used Lemma \ref{lemma:3} and $C_2=C^{\stoc,\bar{\alphab}}_{\error}\left(\prod_{i=1}^N \frac{e^{2\widetilde{g}_i}}{\widetilde{g}^2_j}\right)$.

Inserting the expression of $\widehat{L}$ into the bounds for the error and work contributions, it is immediate to verify that $\text{Work}\big[\mathcal{M}_{\I(\widehat{L})}(u)\big]\leq W_{\max}$ and
\begin{align*}
\text{Error}\big[\mathcal{M}_{\I(\widehat{L})}(u)\big]&\leq C_2 e^{-\sqrt[2N]{\frac{W_{\max}(2N)!}{C_1}}}\left(\sqrt[2N]{\frac{W_{\max}(2N)!}{C_1}}+1\right)^{2N-1}.
\end{align*}
\end{proof}

\subsection{Proof of Theorem \ref{thm:stoc_det}}\label{appendix:thm:stoc_deter}
\begin{theorem}
Let $r_j=\log(2)\widetilde{r}_j$, $\gamma_j=\log(2)\widetilde{\gamma}_j$, $j=1,\dots,N$, $\mathbf{\Theta}=(\frac{\gamma_1}{r_1+\gamma_1},\dots,\frac{\gamma_D}{r_D+\gamma_D})$, $\mu=\min_n\frac{r_n}{\gamma_n}$, $\chi=\max_n \mathbf{\Theta}_n$, and $n(\mathbf{\Theta},\chi):=\#\left\{n: \mathbf{\Theta}_n=\chi\right\}$.
There exists a constant $C_W$ such that for any $W_{\max}$ satisfying $W_{\max}\geq C_W e^{\chi}$, and setting
\begin{equation}\label{eq:choice_L}
L=L(W_{\max})=\frac{1}{\chi}\left(\log\left(\frac{W_{\max}}{C_W}\right)-(n(\mathbf{\Theta},\chi)-1)\log\left(\frac{1}{\chi}\log\left(\frac{W_{\max}}{C_W}\right)\right)\right), 
\end{equation} 
the combination technique solution satisfies
\begin{align}
\text{Work}[\mathcal{M}_{\I(L(W_{\max}))}]\leq W_{\max},\label{eq:workresult}\\
\limsup_{W_{\max}\rightarrow \infty} \frac{Error[\mathcal{M}_{\I(L(W_{\max}))}]}{W_{\max}^{-\mu}(\log(W_{\max})^{(\mu+1)(n(\mathbf{\Theta},\chi)-1})}=C<\infty\label{eq:errorresult},
\end{align} 
\end{theorem}

\begin{proof}
Due to Assumption \ref{ass:work}, the total work associated to $\I(L)$ is bounded by
\begin{align*}
&\text{Work}\big[\mathcal{M}_{\I(L)}(u)\big])=\sum_{(\alphab,\betab)\in \I(L)} \Delta W_{\alphab,\betab} \\
& \leq C_{\work} \sum_{\left\{(\alphab,\betab)\in \mathbb{N}_{+}^{D+N}:\ (\rb+\gammab)\cdot \alphab+\widetilde{\gb}\cdot (\betab+\unob) +|\log(\betab+\unob)|\leq L\right\}} e^{\gammab\cdot \alphab + \log |\betab+\unob|}\\
&\leq C_{\work} \int_{\left\{(\alphab,\betab)\in (1,\infty)^{D+N}:\ (\rb+\gammab)\cdot (\alphab-\unob)+\widetilde{\gb}\cdot \betab +|\log(\betab)|\leq L\right\}} e^{\gammab\cdot \alphab + \log |\betab+\unob|}d\alphab d\betab\\
&= C_{\work}\left(\prod_{i=1}^D \frac{e^{\gamma_i}}{r_i+\gamma_i}\right) \int_{\left\{(\bar{\alphab},\betab)\in (0,\infty)^{D}\times (1,\infty)^{N}:\ |\bar{\alphab}|+\widetilde{\gb}\cdot \betab +|\log(\betab)|\leq L\right\}} e^{\mathbf{\Theta}\cdot  \bar{\alphab} + \log |\betab+\unob|}d\bar{\alphab} d\betab,
\end{align*}
where we used Lemma \ref{lemma:1}, performed the change of variable $\bar{\alpha}_i=(\alpha_i-1)(r_i+\gamma_i)$ and defined $\Theta_i=\frac{\gamma_i}{r_i+\gamma_i}$.
The change of variable $\bar{\beta}_i=\beta_i\widetilde{g}_i$ leads to
\begin{align}\label{eq:Thm2_intermediate}
\medmuskip=-1mu
\thinmuskip=-1mu
\thickmuskip=-1mu
\nulldelimiterspace=0.9pt
\scriptspace=0.9pt 
\arraycolsep0.9em 
\text{Work}\big[\mathcal{M}_{\I(L)}(u)\big]\leq  D_{\work} \int_{\left\{(\bar{\alphab},\bar{\betab})\in (0,\infty)^{D}\times \otimes_{i=1}^N (\widetilde{g}_i,\infty):\ |\bar{\alphab}|+|\bar{\betab}| + |\log(\frac{\bar{\betab}}{\widetilde{\gb}})|\leq L\right\}} e^{\mathbf{\Theta}\cdot  \bar{\alphab} + |\log\left(\frac{\bar{\betab}}{\widetilde{\gb}}+\bm{1}\right)|} d\bar{\alphab} d\bar{\betab},
\end{align}
where $D_{\work}:=C_{\work}\left(\prod_{i=1}^D \frac{e^{\gamma_i}}{r_i+\gamma_i}\right)\left(\prod_{i=1}^N\frac{1}{\widetilde{g}_i}\right)$.
Denoting with $\text{Int}$ the right hand side of \eqref{eq:Thm2_intermediate} and using Lemma \ref{lemma:4} we get

\begin{align}\label{eq:Thm2_intermediate2}
\medmuskip=-1mu
\thinmuskip=-1mu
\thickmuskip=-1mu
\nulldelimiterspace=0.9pt
\scriptspace=0.9pt 
\arraycolsep0.9em 
&\text{Int}=\int_{\left\{\bar{\betab}\in\otimes_{i=1}^N (\widetilde{\gb}_i,\infty):\ |\bar{\betab}| + |\log(\frac{\bar{\betab}}{\widetilde{\gb}}))|\leq L\right\}} e^{|\log\left(\frac{\bar{\betab}}{\widetilde{\gb}}+\unob\right)|}\int_{\left\{\bar{\alphab} \in (0,\infty)^D:\ |\alphab|\leq  L-|\bar{\betab}|-|\log(\frac{\bar{\betab}}{\widetilde{\gb}})|\right\}} e^{\mathbf{\Theta}\cdot \bar{\alphab}} d\bar{\alphab} d\bar{\betab}\nonumber\\
&\leq \int_{\left\{\bar{\betab}\in\otimes_{i=1}^N (\widetilde{\gb}_i,\infty):\ |\bar{\betab}| + |\log(\frac{\bar{\betab}}{\widetilde{\gb}})|\leq L\right\}} \mathcal{U}_D(\mathbf{\Theta},0)e^{|\log\left(\frac{\betab}{\widetilde{\gb}}+\unob\right)|+\chi(L-|\bar{\betab}|-|\log(\frac{\bar{\betab}}{\widetilde{\gb}})|)}L^{\eta(\Theta,\chi)-1}d\bar{\betab}\nonumber\\
&\leq \mathcal{U}_D(\mathbf{\Theta},0) e^{\chi L}\int_{\left\{\bar{\betab}\in\otimes_{i=1}^N (\widetilde{\gb}_i,\infty):\ |\bar{\betab}| + |\log(\frac{\bar{\betab}}{\widetilde{\gb}})|\leq L\right\}} e^{|\log\left(\frac{\bar{\betab}}{\widetilde{\gb}}+\unob\right)|}e^{-\chi(|\bar{\betab}|+|\log(\frac{\bar{\betab}}{\widetilde{\gb}})|)}L^{\eta(\Theta,\chi)-1}d\bar{\betab}\nonumber\\
&\leq \mathcal{U}_D(\mathbf{\Theta},0) e^{\chi L}L^{\eta(\Theta,\chi)-1}\int_{\left\{\bar{\betab}\in\otimes_{i=1}^N (\widetilde{\gb}_i,\infty):\ |\bar{\betab}| + |\log(\frac{\bar{\betab}}{\widetilde{\gb}})|\leq L\right\}} e^{|\log\left(\frac{\bar{\betab}}{\widetilde{\gb}}+\unob\right)|}e^{-\chi(|\bar{\betab}|+|\log(\frac{\bar{\betab}}{\widetilde{\gb}})|)}d\bar{\betab}
\end{align}
Plugging \eqref{eq:Thm2_intermediate2} into \eqref{eq:Thm2_intermediate}

we conclude that 
\begin{align*}
\text{Work}\big[\mathcal{M}_{\I(L)}(u)\big]\leq C_{W}L^{\eta(\mathbf{\Theta},\chi)-1}e^{\chi L},
\end{align*}
where $C_W=D_{\work}\mathcal{U}_D(\mathbf{\Theta},0)\int_{\left\{\bar{\betab}\in\otimes_{i=1}^N (\widetilde{\gb}_i,\infty):\ |\bar{\betab}| + |\log(\frac{\bar{\betab}}{\widetilde{\gb}})|\leq L\right\}} e^{|\log(\frac{\bar{\betab}}{\widetilde{\gb}}+\unob)|} e^{-\chi(|\betab|+|\log(\frac{\bar{\betab}}{\widetilde{\gb}})|)}d\bar{\betab}<\infty$ as the latter integral is bounded for every $L$.
We now focus on the error estimate.
\begin{equation}
\begin{aligned}\label{eq:error_estimate}
\text{Error}\big[\mathcal{M}_{\I(\widehat{L})}(u)\big]&\leq C_{\error}\sum_{\left\{(\alphab,\betab)\in \mathbb{N}_{+}^{D+N}:\ (\rb+\gammab)\cdot \alphab + \widetilde{\gb}\cdot (\betab+\unob) + |\log(\betab+\unob)|>L\right\}} e^{-\rb\cdot\alphab -\widetilde{\gb}\cdot (\betab+\unob)}\\
& =C_{\error}\bigg(\sum_{\left\{(\alphab,\betab)\in \mathbb{N}_{+}^{D+N}:\ (\rb+\gammab)\cdot \alphab>L\right\}}e^{-\rb\cdot\alphab -\widetilde{\gb}\cdot (\betab+\unob)}\\
 &+\sum_{\left\{\alphab\in \mathbb{N}_{+}^D (\rb+\gammab)\cdot \alphab \leq L\right\}} e^{-\rb\cdot\alphab}\sum_{\left\{\betab\in \mathbb{N}_{+}^N \widetilde{\gb}\cdot (\betab+\unob) + |\log(\betab+\unob)|> L -(\rb+\gammab)\cdot\alphab \right\}} e^{ -\widetilde{\gb}\cdot (\betab+\unob)}.\bigg)
\end{aligned}
\end{equation}
Using the change of variable $\bar{\alpha}_i=(r_i+\gamma_i)\alpha_i$, the first term can be bounded as
\begin{align}\label{eq:first_estimate_final}
\sum_{\left\{(\alphab,\betab)\in \mathbb{N}_{+}^{D+N}:\ (\rb+\gammab)\cdot \alphab>L\right\}}&e^{-\rb\cdot\alphab -\widetilde{\gb}\cdot (\betab+\unob)}=\left(\sum_{\betab\in \mathbb{N}_{+}^N} \prod_{i=1}^N e^{-\widetilde{g}_j(\beta_j+1)}\right)\left(\sum_{\alphab \in \mathbb{N}_{+}^N:\ (\rb+\gammab)\cdot \alphab>L} e^{-\rb\cdot\alphab}\right)\nonumber\\
&= \left(\prod_{i=1}^N \frac{e^{-2\widetilde{g}_i}}{1-e^{-\widetilde{g}_i}}\right)\left(\sum_{\alphab \in \mathbb{N}_{+}^N:\ (\rb+\gammab)\cdot \alphab>L} e^{-\rb\cdot\alphab}\right)\nonumber\\
&\leq \left(\prod_{i=1}^N \frac{e^{-2\widetilde{g}_i}}{1-e^{-\widetilde{g}_i}}\right)\int_{\alphab \in (1,\infty)^D:\ (\rb+\gammab)\cdot \alphab>L} e^{-\rb\cdot\alphab+\rb}d\alphab\nonumber\\
&\leq \left(\prod_{i=1}^N \frac{e^{-2\widetilde{g}_i}}{1-e^{-\widetilde{g}_i}}\right)\left(\prod_{i=1}^D \frac{e^{r_i}}{r_i+\gamma_i}\right) \int_{\bar{\alphab} \in \otimes_{i=1}^D(r_i+\gamma_i,\infty):\ |\bar{\alphab}|>L} e^{-\mathbf{\Phi}\cdot \bar{\alphab}}d\bar{\alphab}\nonumber\\
& \leq C_{E,1} e^{-\delta L}L^{n(\Phi,\delta)-1},
\end{align}
where we used Lemma \ref{lemma:5} with 
\[\mathbf{\Phi}_i=\frac{r_i}{r_i+\gamma_i},\quad \delta=\min_{i=1,\dots,N} \mathbf{\Phi}_i,\quad \text{and}\quad C_{E,1}:=\left(\prod_{i=1}^N \frac{e^{-2\widetilde{g}_i}}{1-e^{-\widetilde{g}_i}}\right)\left(\prod_{i=1}^D \frac{e^{r_i}}{r_i+\gamma_i}\right)  \mathcal{B}_D(\mathbf{\Phi}).\]

Concerning the second term, we first define $H:=(\rb+\gammab)\cdot \alphab$ and $\widehat{L}:=L-H$, so that
\begin{align*}
&\sum_{\left\{\betab\in \mathbb{N}_{+}^N:\ \widetilde{\gb}\cdot (\betab+\unob)+|\log(\betab+\unob)|>\widehat{L} \right\}}e^{-\widetilde{\gb}\cdot (\betab+\unob)}
\leq \int_{\left\{\betab\in \otimes_{i=1}^N(1,\infty):\ \widetilde{\gb}\cdot (\betab+\unob)+|\log(\betab+\unob)|>\widehat{L} \right\}}e^{-\widetilde{\gb}\cdot \betab}d\betab\nonumber\\
&= e^{|\widetilde{\gb}|}\int_{\left\{\betab\in \otimes_{i=1}^N(1,\infty):\ \widetilde{\gb}\cdot (\betab+\unob) +|\log(\betab+\unob)|>\widehat{L} \right\}}e^{-\widetilde{\gb}\cdot (\betab+\unob)+|\log(\betab+\unob)|-|\log(\betab+\unob)|}d\betab
\end{align*}
Setting $t_i=\widetilde{g}_i(\beta_i+1)+\log(\beta_i+1)=t_i$, $q(t_i)=\beta_i$, with $q(t_i)\leq \frac{t_i}{\widetilde{g}_i}-1$, leads to
{\footnotesize
\begin{align}\label{eq:intermediate_estimate}
\sum_{\left\{\betab\in \mathbb{N}_{+}^N:\ \widetilde{\gb}\cdot \betab+|\log(\betab+\unob)|>\widehat{L} \right\}}e^{-\widetilde{\gb}\cdot (\betab+\unob)}&\leq \left(\prod_{i=1}^N e^{\widetilde{g}_i}\right)  \int_{\left\{\mathbf{t}\in \otimes_{i=1}^D (2\widetilde{g}_i+\log(2),\infty):\ |\mathbf{t}|>\widehat{L}\right\}} e^{-|\mathbf{t}|}\left(\prod_{i=1}^N \frac{(q(t_i)+1)^2}{\widetilde{g}_j(q(t_i)+1)}\right)d\mathbf{t}\nonumber\\
&\leq \left(\prod_{i=1}^N \frac{e^{\widetilde{g}_i}}{\widetilde{g}_i^2}\right)
\int_{\left\{\mathbf{t}\in \otimes_{i=1}^D (2\widetilde{g}_i+\log(2),\infty):\ |\mathbf{t}|>\widehat{L}\right\}} e^{-|\mathbf{t}|+|\log(\mathbf{t})|}d\mathbf{t}\nonumber\\
&\leq C_{E,2} e^{-\widehat{L}}(\widehat{L}+1)^{2N-1},
\end{align}
}\normalsize
where we used Lemma \ref{lemma:3} in the last step and set $C_{E,2}=\left(\prod_{i=1}^N \frac{e^{\widetilde{g}_i}}{\widetilde{g}_i^2}\right)$.
Inserting \eqref{eq:intermediate_estimate} into the second term of \eqref{eq:error_estimate} and using Lemma \ref{lemma:4},
\begin{equation*}\label{eq:second_estimate_final}
\begin{aligned}
&\sum_{\left\{\alphab\in \mathbb{N}_{+}^D (\rb+\gammab)\cdot \alphab \leq L\right\}} e^{-\rb\cdot\alphab}\sum_{\left\{\betab\in \mathbb{N}_{+}^N \widetilde{\gb}\cdot (\betab+\unob) + |\log(\betab+\unob)|> L -(\rb+\gammab)\cdot\alphab \right\}} e^{ -\widetilde{\gb}\cdot (\betab+\unob)}\\
&\leq C_{E,2}\sum_{\left\{\alphab\in \mathbb{N}_{+}^D (\rb+\gammab)\cdot \alphab \leq L\right\}} e^{-L+\gammab\cdot\alphab}(L+1-(\rb+\gammab)\cdot\alphab)^{2N-1}\\
&\leq C_{E,2}\int_{\left\{\alphab\in (1,\infty)^D:\ (\rb+\gammab)\cdot (\alphab-\unob) \leq L\right\}} e^{-L+\gammab\cdot\alphab}(L+1-(\rb+\gammab)\cdot(\alphab-\unob))^{2N-1}\\
&\leq C_{E,2}\left(\prod_{i=1}^D \frac{e^{\gamma_i}}{\gamma_i+r_i}\right)e^{-L} \int_{\left\{\mathbf{x}\in (0,\infty)^D:\ |\mathbf{x}| \leq L\right\}} e^{ \Theta\cdot\mathbf{x}}(L+1-|\mathbf{x}|)^{2N-1}d\mathbf{x}\\
&\leq C_{E,3}e^{(\chi-1)L}L^{n(\mathbf{\Theta},\chi)-1},
\end{aligned}
\end{equation*}
where $C_{E,3}=C_{E,2}\left(\prod_{i=1}^D \frac{e^{\gamma_i}}{\gamma_i+r_i}\right) \mathcal{U}_D(\mathbf{\Theta},2N-1)$.
Since $\delta=1-\chi$, $n(\mathbf{\Theta},\chi)=n(\mathbf{\Phi},\delta)$ and using the bounds for the two terms of \eqref{eq:error_estimate}, we conclude 
\[\text{Error}\big[\mathcal{M}_{\I(L)}(u)\big]\leq C_{E,1}e^{-\delta L}L^{\eta(\mathbf{\Phi},\delta)-1}+C_{E,3} e^{(\chi-1)L}L^{n(\mathbf{\Theta},\chi)-1}=C_{E,4}e^{(\chi-1) L}L^{n(\mathbf{\Theta},\chi)-1}.\]
Finally, a direct calculation shows that \eqref{eq:choice_L} leads to \eqref{eq:workresult} and \eqref{eq:errorresult}.
\end{proof}


\begin{thebibliography}{10}

\bibitem{babuvska2010stochastic}
I.~Babu{\v{s}}ka, F.~Nobile, and R.~Tempone.
\newblock A stochastic collocation method for elliptic partial differential
  equations with random input data.
\newblock {\em SIAM review}, 52(2):317--355, 2010.

\bibitem{back2011stochastic}
J.~B{\"a}ck, F.~Nobile, L.~Tamellini, and R.~Tempone.
\newblock Stochastic spectral {G}alerkin and collocation methods for {PDE}s
  with random coefficients: a numerical comparison.
\newblock In {\em Spectral and high order methods for partial differential
  equations}, pages 43--62. Springer, 2011.

\bibitem{beck2019iga}
J.~Beck, L.~Tamellini, and R.~Tempone.
\newblock {IGA}-based multi-index stochastic collocation for random {PDE}s on
  arbitrary domains.
\newblock {\em Computer Methods in Applied Mechanics and Engineering},
  351:330--350, 2019.

\bibitem{brezzi1980finite}
F.~Brezzi, J.~Rappaz, and P.-A. Raviart.
\newblock Finite dimensional approximation of nonlinear problems: Part i:
  Branches of nonsingular solutions.
\newblock {\em Numerische Mathematik}, 36(1):1--25, 1980.

\bibitem{bungartz2004sparse}
H.-J. Bungartz and M.~Griebel.
\newblock Sparse grids.
\newblock {\em Acta numerica}, 13:147--269, 2004.

\bibitem{chkifa2014high}
A.~Chkifa, A.~Cohen, and C.~Schwab.
\newblock High-dimensional adaptive sparse polynomial interpolation and
  applications to parametric {PDE}s.
\newblock {\em Foundations of Computational Mathematics}, 14:601--633, 2014.

\bibitem{cohen2015approximation}
A.~Cohen and R.~DeVore.
\newblock Approximation of high-dimensional parametric {PDE}s.
\newblock {\em Acta Numerica}, 24:1--159, 2015.

\bibitem{cohen2010convergence}
A.~Cohen, R.~DeVore, and C.~Schwab.
\newblock Convergence rates of best n-term {G}alerkin approximations for a
  class of elliptic {PDE}s.
\newblock {\em Foundations of Computational Mathematics}, 10(6):615--646, 2010.

\bibitem{cohen2011analytic}
A.~Cohen, R.~Devore, and C.~Schwab.
\newblock Analytic regularity and polynomial approximation of parametric and
  stochastic elliptic {PDE}s.
\newblock {\em Analysis and Applications}, 9(01):11--47, 2011.

\bibitem{ern2004theory}
A.~Ern and J.L. Guermond.
\newblock {\em Theory and Practice of Finite Elements}.
\newblock Applied Mathematical Sciences. Springer New York, 2004.

\bibitem{ernst2018convergence}
O.~G. Ernst, B.~Sprungk, and L.~Tamellini.
\newblock Convergence of sparse collocation for functions of countably many
  gaussian random variables (with application to elliptic pdes).
\newblock {\em SIAM Journal on Numerical Analysis}, 56(2):877--905, 2018.

\bibitem{gerstner2003dimension}
T.~Gerstner and M.~Griebel.
\newblock Dimension--adaptive tensor--product quadrature.
\newblock {\em Computing}, 71(1):65--87, 2003.

\bibitem{combination_te}
M.~Griebel and H.~Harbrecht.
\newblock On the convergence of the combination technique.
\newblock In {\em Sparse Grids and Applications - Munich 2012}, pages 55--74,
  Cham, 2014. Springer International Publishing.

\bibitem{griebel1990combination}
M.~Griebel, M.~Schneider, and C.~Zenger.
\newblock A combination technique for the solution of sparse grid problems.
\newblock {\em Iterative Methods in Linear Algebra}, pages 263--281, 1992.

\bibitem{guignard2018posteriori}
D.~Guignard and F.~Nobile.
\newblock A posteriori error estimation for the stochastic collocation finite
  element method.
\newblock {\em SIAM Journal on Numerical Analysis}, 56(5):3121--3143, 2018.

\bibitem{doi:10.1137/19M1294952}
P.~A. Guth, V.~Kaarnioja, F.~Y. Kuo, C.~Schillings, and I.~H. Sloan.
\newblock A {Q}uasi-{M}onte {C}arlo method for optimal control under
  uncertainty.
\newblock {\em SIAM/ASA Journal on Uncertainty Quantification}, 9(2):354--383,
  2021.

\bibitem{haji2016multi2}
A.-L. Haji-Ali, F.~Nobile, L.~Tamellini, and R.~Tempone.
\newblock Multi-index stochastic collocation convergence rates for random
  {PDE}s with parametric regularity.
\newblock {\em Foundations of Computational Mathematics}, 16(6):1555--1605,
  2016.

\bibitem{haji2016multi}
A.-L. Haji-Ali, F.~Nobile, L.~Tamellini, and R.~Tempone.
\newblock Multi-index stochastic collocation for random {PDE}s.
\newblock {\em Computer Methods in Applied Mechanics and Engineering},
  306:95--122, 2016.

\bibitem{MIMC}
A.-L. Haji-Ali, F.~Nobile, and R.~Tempone.
\newblock Multi-index {M}onte {C}arlo: when sparsity meets sampling.
\newblock {\em Numerische Mathematik}, 132(4):767--806, 2016.

\bibitem{harbrecht2012multilevel}
H.~Harbrecht, M.~Peters, and M.~Siebenmorgen.
\newblock On multilevel quadrature for elliptic stochastic partial differential
  equations.
\newblock In {\em Sparse grids and applications}, pages 161--179. Springer,
  2012.

\bibitem{hinze2008optimization}
M.~Hinze, R.~Pinnau, M.~Ulbrich, and S.~Ulbrich.
\newblock {\em Optimization with PDE constraints}, volume~23.
\newblock Springer Science \& Business Media, 2008.

\bibitem{kouri2012approach}
D.~P. Kouri.
\newblock {\em An approach for the adaptive solution of optimization problems
  governed by partial differential equations with uncertain coefficients}.
\newblock Rice University, 2012.

\bibitem{kouri2013trust}
D.~P. Kouri, M.~Heinkenschloss, D.~Ridzal, and B.~G. van Bloemen~Waanders.
\newblock A trust-region algorithm with adaptive stochastic collocation for
  {PDE} optimization under uncertainty.
\newblock {\em SIAM Journal on Scientific Computing}, 35(4):A1847--A1879, 2013.

\bibitem{kouri2018inexact}
D.~P. Kouri and D.~Ridzal.
\newblock Inexact trust-region methods for {PDE}-constrained optimization.
\newblock In {\em Frontiers in PDE-Constrained Optimization}, pages 83--121.
  Springer, 2018.

\bibitem{kouri2018optimization}
D.~P. Kouri and A.~Shapiro.
\newblock Optimization of {PDE}s with uncertain inputs.
\newblock In {\em Frontiers in PDE-Constrained Optimization}, pages 41--81.
  Springer, 2018.

\bibitem{doi:10.1137/16M1086613}
D.~P. Kouri and T.~M. Surowiec.
\newblock Existence and optimality conditions for risk-averse {PDE}-constrained
  optimization.
\newblock {\em SIAM/ASA Journal on Uncertainty Quantification}, 6(2):787--815,
  2018.

\bibitem{kroese2013handbook}
D.P. Kroese, T.~Taimre, and Z.I. Botev.
\newblock {\em Handbook of {M}onte {C}arlo Methods}.
\newblock Wiley Series in Probability and Statistics. Wiley, 2013.

\bibitem{kunoth2013analytic}
A.~Kunoth and C.~Schwab.
\newblock Analytic regularity and {GPC} approximation for control problems
  constrained by linear parametric elliptic and parabolic {PDE}s.
\newblock {\em SIAM Journal on Control and Optimization}, 51(3):2442--2471,
  2013.

\bibitem{kunoth2016sparse}
A.~Kunoth and C.~Schwab.
\newblock Sparse adaptive tensor {G}alerkin approximations of stochastic
  {PDE}-constrained control problems.
\newblock {\em SIAM/ASA Journal on Uncertainty Quantification},
  4(1):1034--1059, 2016.

\bibitem{lord_powell_shardlow_2014}
G.~J. Lord, C.~E. Powell, and T.~Shardlow.
\newblock {\em An Introduction to Computational Stochastic {PDE}s}.
\newblock Cambridge Texts in Applied Mathematics. Cambridge University Press,
  2014.

\bibitem{martello1990knapsack}
S.~Martello and P.~Toth.
\newblock {\em Knapsack problems: algorithms and computer implementations}.
\newblock John Wiley \& Sons, Inc., 1990.

\bibitem{Matthieu}
M.~Martin, S.~Krumscheid, and F.~Nobile.
\newblock Complexity analysis of stochastic gradient methods for
  {PDE}-constrained optimal control problems with uncertain parameters.
\newblock {\em ESAIM: M2AN}, 55(4):1599--1633, 2021.

\bibitem{martinez2018optimal}
J.~Mart{\'\i}nez-Frutos and F.~P. Esparza.
\newblock {\em Optimal Control of {PDE}s Under Uncertainty: An Introduction
  with Application to Optimal Shape Design of Structures}.
\newblock Springer, 2018.

\bibitem{nobile2016convergence}
F.~Nobile, L.~Tamellini, and R.~Tempone.
\newblock Convergence of quasi-optimal sparse-grid approximation of
  {H}ilbert-space-valued functions: application to random elliptic {PDE}s.
\newblock {\em Numerische Mathematik}, 134(2):343--388, 2016.

\bibitem{vanzan}
F.~Nobile and T.~Vanzan.
\newblock Preconditioners for robust optimal control problems under
  uncertainty.
\newblock {\em Numerical Linear Algebra with Applications}, page e2472, 2022.

\bibitem{pflaum}
Christoph P.
\newblock {\em Diskretisierung elliptischer Differentialgleichungen mit
  d\"{u}unnen Gittern}.
\newblock PhD thesis, 1996.

\bibitem{peherstorfer2018survey}
B.~Peherstorfer, K.~Willcox, and M.~Gunzburger.
\newblock Survey of multifidelity methods in uncertainty propagation,
  inference, and optimization.
\newblock {\em Siam Review}, 60(3):550--591, 2018.

\bibitem{reichel1990newton}
L.~Reichel.
\newblock Newton interpolation at {L}eja points.
\newblock {\em BIT}, 30(2):332--346, 1990.

\bibitem{doi:10.1137/16M1082561}
Pieterjan Robbe, Dirk Nuyens, and Stefan Vandewalle.
\newblock A multi-index quasi--monte carlo algorithm for lognormal diffusion
  problems.
\newblock {\em SIAM Journal on Scientific Computing}, 39(5):S851--S872, 2017.

\bibitem{schwab2011sparse}
C.~Schwab and C.~J. Gittelson.
\newblock Sparse tensor discretizations of high-dimensional parametric and
  stochastic {PDE}s.
\newblock {\em Acta Numerica}, 20:291--467, 2011.

\bibitem{seidler2022dimension}
U.~Seidler and M.~Griebel.
\newblock A dimension-adaptive combination technique for uncertainty
  quantification.
\newblock {\em arXiv preprint arXiv:2204.05574}, 2022.

\bibitem{shapiro2021lectures}
A.~Shapiro, D.~Dentcheva, and A.~Ruszczynski.
\newblock {\em Lectures on stochastic programming: modeling and theory}.
\newblock SIAM, 2021.

\bibitem{teckentrup2015multilevel}
A.~Teckentrup, P.~Jantsch, C.~G Webster, and M.~Gunzburger.
\newblock A multilevel stochastic collocation method for partial differential
  equations with random input data.
\newblock {\em SIAM/ASA Journal on Uncertainty Quantification},
  3(1):1046--1074, 2015.

\bibitem{trefethen2008gauss}
L.~N. Trefethen.
\newblock Is {G}auss quadrature better than {C}lenshaw--{C}urtis?
\newblock {\em SIAM review}, 50(1):67--87, 2008.

\bibitem{van2019robust}
A.~Van~Barel and S.~Vandewalle.
\newblock Robust optimization of {PDE}s with random coefficients using a
  multilevel {M}onte {C}arlo method.
\newblock {\em SIAM/ASA Journal on Uncertainty Quantification}, 7(1):174--202,
  2019.

\end{thebibliography}
\end{document}